\documentclass{article}

\usepackage{epstopdf}
\usepackage[caption=false]{subfig}

\usepackage[numbers,sort&compress]{natbib}
\bibpunct[, ]{[}{]}{,}{n}{,}{,}

\usepackage{amsmath}
\usepackage{amsthm}
\usepackage{amssymb}

\theoremstyle{plain}
\newtheorem{theorem}{Theorem}[section]
\newtheorem{lemma}[theorem]{Lemma}
\newtheorem{corollary}[theorem]{Corollary}
\newtheorem{proposition}[theorem]{Proposition}

\usepackage{empheq} 
\usepackage{framed} 
\usepackage[utf8]{inputenc}
\usepackage{color}

\renewcommand{\color}[1]{}

\newcommand{\wh}{\widehat}
\newcommand{\ol}{\overline}

\newcommand{\R}{\mathbb{R}}
\newcommand{\eps}{\varepsilon}

\newcommand{\sign}{\,{\rm{sign}}\,}
\newcommand{\diag}{\,{\rm{diag}}\,}
\newcommand{\inv}{{\rm{inv}}}
\newcommand{\lw}{{\rm{lower}}}
\newcommand{\up}{{\rm{upper}}}

\newcommand{\Arg}{{\rm{Arg}}\,}

\begin{document}

\title{New versions of Newton method: step-size choice, convergence domain and under-determined equations}

\author{
{Boris Polyak 
and Andrey Tremba}\\
{Institute for Control Sciences, Moscow, Russia}
}

\maketitle

\begin{abstract}Newton method is one of the most powerful methods for finding solutions of nonlinear equations and for proving their existence. In its ``pure'' form it has fast convergence near the solution, but small convergence domain.
On the other hand damped Newton method has slower convergence rate, but weaker conditions on the initial point.
We provide new versions of Newton-like algorithms, resulting in combinations of Newton and damped Newton method with special step-size choice, and estimate its convergence domain. Under some assumptions the convergence is global. Explicit complexity results are also addressed.
The adaptive version of the algorithm (with no a priori constants knowledge) is presented. 
The method is applicable for under-determined equations (with $m<n$, $m$ being the number of equations and $n$ being the number of variables).
The results are specified for systems of quadratic equations, for composite mappings and for one-dimensional equations and inequalities. 
	

	
\end{abstract}

\section{Introduction}
\label{intro}
Consider nonlinear equation
\begin{equation} \label{eq:main-nonlinear-P} 
P(x) = 0, 
\end{equation}
written via the vector function $P : \R^n\rightarrow \R^m$. 
There exists the huge bunch of literature on solvability of such equations and numerical methods for their solution, see e.g. the classical monographs \cite{OR,DS}. One of the most powerful methods is \emph{Newton method}, going back to such giants as Newton, Cauchy, Fourier. 
The general form of the method is due to Kantorovich \cite{Kant82}; on history and references see \cite{Kelley,Yam,Polyak2006,Deuflhard-2004}. The basic version of Newton method for \eqref{eq:main-nonlinear-P} is applicable when $P(x)$ is differentiable and $P'(x)$ is invertible (this implies $m=n$):
\begin{equation} \label{eq:classic-newton}
x^{k+1} = x^k - P'(x^k)^{-1} P(x^k).
\end{equation}
The method converges under some natural conditions, moreover it can be used for obtaining existence theorems for the solution (see references cited above). Unfortunately Newton method converges only locally: it requires a good initial approximation $x^0$ (so called ``hot start''). Convergence conditions can be relaxed for \emph{damped Newton method}
\begin{equation}
\label{damped} \nonumber
x^{k+1} = x^k - \alpha P'(x^k)^{-1} P(x^k)
\end{equation}
with $0<\alpha < 1$. 

The advanced, generalized writing of Newton method is
\begin{equation}
\label{eq:newton-with-alpha_k-for-P} 
\begin{array}{l}
x^{k+1} = x^k - \alpha_k z^k, \;\; k = 0, 1, \ldots \\
z^k \in \Arg \min_z \{||z||: P'(x^k)z = P(x^k)\}.
\end{array}
\end{equation}
This variant relies on the solvability of the linear equation only, and it also admits non-constant step-size. It is applicable to under-determined systems of equations ($m < n$) and to non-linear equations in Banach space. The latter are outside of the scope of this paper but its analysis is essentially the same. 

If $m=n$ and $P'(x^k)^{-1}$ exists, the method \eqref{eq:newton-with-alpha_k-for-P} coincides with classical Newton method for $\alpha_k = 1$ and damped Newton method for $\alpha_k = \alpha < 1$. 
Starting at some initial point $x^0$ the latter method converges to $x^*$ under some additional constraints on residual $||P(x^0)||$ and function-related constants $L$, $\rho$, $\mu$, $\alpha$ (see Theorems below for rigorous conditions). 

In explicit form Newton method for $m\neq n$ has been written by Ben-Israel \cite{BI}:
\begin{equation} \label{eq:newton-with-pseudoinverse}
x^{k+1} = x^k - P'(x^k)^\dagger P(x^k), 
\end{equation} 
where
$A^\dagger$ stands for Moore-Penrose pseudoinverse of $A$. However the results in
\cite{BI} are mostly oriented on over-determined systems, and the assumptions of the theorems in \cite{BI} are hard to check. 
Other publications on under-determined equations include \cite{Husler,Nesterov,Polyak-1964,Walker,Nashed-Chen}. Moreover there exist numerous literature on more general settings: equalities plus inequalities \cite{Pshen,Rob}, optimization problems \cite{Burke,Hager} with more general algorithms, which can be applied to solving of equations as particular case.
In next Section~\ref{sec:under-determined} we discuss the under-determined finite-dimensional case in more details.

There is a very similar problem statement to \eqref{eq:main-nonlinear-P}, made in terms of the equation with the variable right-hand side
\begin{equation} \label{eq:main-nonlinear-y} 
g(x) = y,
\end{equation} 
with $g: \R^n \rightarrow \R^m$ having a known solution $\ol{x} : g(\ol{x}) = 0$ and the variable $y$ as a parameter.

The question is: for which right-hand side part $y$ the equation is still feasible and what are the solutions?
This problem arises in finding image of a mapping $\{g(x) : x \in \R^n\}$, checking robustness/sensitivity of a solution or exploration problem of the image, etc. 
In general, this problem is hardly solvable, but we can provide \emph{local} sufficient conditions of feasibility, imposed on $y$. 

Trivially, equation \eqref{eq:main-nonlinear-y} can be written in the form \eqref{eq:main-nonlinear-P} with $P_y(x) = g(x) - y = 0$ and $\|P_y(\ol{x})\| = \|y\|$.
Thus conditions on feasibility of the equation $P_y(x) = 0$ are exactly conditions on feasibility \eqref{eq:main-nonlinear-y} with respect to right-hand side $y$, and vice-versa. 

Let us explain the connection between \eqref{eq:main-nonlinear-P} and \eqref{eq:main-nonlinear-y} deeper. There are few approaches to treating feasibility of an equation. One of them is to prove existence of the solution by providing \emph{semi-local} existence theorems. These involve conditions in some point and/or around it, and prove that if such conditions hold, then a solution exists.
It is not necessary to provide tools for finding this solution.

Another way is to explore a constructive, algorithmic way of solving the equation, starting at a point $x^0$, resulting in a sequence $\{x^k\}$, and prove convergence to a solution $x^k \rightarrow x^*$ (e.g. fixed-point theorems). The convergence conditions are typically tied to the sequence, including starting point $x^0$. The conditions do not necessarily coincide with the conditions of semi-local existence theorems. Moreover, the conditions ensuring \emph{faster} convergence of the algorithm are typically more strict, than the conditions of the semi-local existence theorems.

In Newton method theory these approaches are closely connected, and semi-local theorems are often proved via convergence of variants of Newton method.
We also show it below in Theorem~\ref{thm:existence-solution-for-g}. 
This relation becomes very clear in comparison of equations \eqref{eq:main-nonlinear-P} and \eqref{eq:main-nonlinear-y}.
Naturally feasibility of \eqref{eq:main-nonlinear-P} being solved by a Newton-like method (i.e. \emph{convergence} of the Newton method started at $x^0$!) is stated in terms of norm of initial residual, say $\|P(x^0)\| \leq s$. In terms of \eqref{eq:main-nonlinear-y} the very same Newton-like method, being applied to the constructed $P_y(\cdot)$ and being started at the same point $\ol{x}=x^0$, converges for \emph{any} fixed $y : \|y\| \leq s$. This analysis claims \emph{feasibility} of \eqref{eq:main-nonlinear-y} for all such $y$.
Therefore finding the \emph{largest} set of possible $y$ is essentially the same as the problem of finding the \emph{broadest} residual range $P(x^0)$. 
Without loss of generality, {\color{red}through} the paper we assume $\ol{x} = 0$, and switch between
problems \eqref{eq:main-nonlinear-P} and \eqref{eq:main-nonlinear-y} as equivalent ones. The difference is clear from context.

We also examine some special cases of the nonlinear equations. One of them is the quadratic case, when all components of $g$ are quadratic functions:
\begin{equation} \label{eq:quadratic-componentwise} 
g_{i}(x) = \frac{1}{2}(A_i x, x) + (b_i,x),\;\; A_i=A_i^T, \;\; b_i\in \R^n. 
\end{equation}
In this case we try to specify above results and design the algorithms to check feasibility of a vector $y\in R^m$.

The first goal of the present paper is to give explicit expressions of the method
\eqref{eq:newton-with-alpha_k-for-P} for various norms and to provide simple, easily checkable conditions
for convergence of the method. 
This also provides existence theorems: what is \emph{a feasible set}
$Y$ such that $y\in Y$ implies solvability of \eqref{eq:main-nonlinear-y}.

The second goal is to develop constructive algorithms for choosing step-sizes $\alpha_k $ to achieve fast
and \emph{as global as possible} convergence. We suggest different strategies for constructing
algorithms, study their properties and provide explicit convergence conditions for the method and demonstrate its potentially global convergence.

The main contributions of the paper are threefold.

\begin{enumerate}
\item We propose the novel procedure for adjusting step-size $\alpha_k$. The strategy guarantees wide range of convergence (in some cases the algorithm converges globally) and fast rate of convergence (local quadratic convergence, typical for pure Newton method). Moreover explicit formula for method's complexity is provided.

\item The choice of norms in the algorithm can be different, thus we arrive to different versions of the algorithm. For instance, Euclidean norms imply explicit form of desired direction $z^k$ (the same as in \eqref{eq:newton-with-pseudoinverse}) while $\ell_1$ norm provides sparse approximations etc.

\item We consider numerous applications, including under-determined cases, quadratic equations, one equation with $n$ variables.
\end{enumerate}

Few words on comparison with known results. 
In the paper \cite{Polyak-1964} results on solvability of nonlinear equations in Banach spaces and on application of Newton-like methods have been formulated in semi-local form. One of the results from \cite{Polyak-1964} adopted to our notation and finite-dimensional case claims that if $P'(x)$ exists and is Lipschitz on a ball $B$ centered in $x^0$ and estimate $||P'(x)^T h||_{\color{red}*} \geq \mu ||h||_{\color{red}*},\; \mu > 0,\; \forall h$ holds on $B$, then equation \eqref{eq:main-nonlinear-P} has a solution $x^*$ provided $||P(x^0)|| < \frac{\rho}{\mu}$.
Another result deals with convergence of Newton method; however the method is not provided in explicit form. The condition on the derivative has extension to non-differentiable functions and multi-valued mapping an is known as \textit{metric regularity}, used for proving existence theorems in different cases.

The paper, which contains the closest results to ours, is \cite{Nesterov}. Nesterov addresses the same problem \eqref{eq:main-nonlinear-P} and his method (in our notation) has the form

\begin{equation} \nonumber \label{eq:nesterov} 
\begin{array}{l}
x^{k+1} = x^k - z^k, \;\; k = 0, 1, \ldots \\
\displaystyle
z^k = \arg \min_z \{||P(x^k) - P'(x^k)z|| + M ||z||^2\}, 
\end{array} 
\end{equation}
where $M$ is the scalar parameter to be adjusted at each iteration. Nesterov's assumptions are close to ours and his results on solvability of equations and on convergence of the method are similar. The main difference is the method itself; it is not clear how to solve the auxiliary optimization problem in Nesterov's method, while finding $z^k$ in our method can be implemented in explicit form. Other papers on under-determined equations mentioned above either do not specify the technique for solving the linearized auxiliary equation, or restrict analysis with Euclidean norm and/or pure Newton step-size $\alpha_k=1$, see
e.g. \cite{Polyak-1964,PT,Walker,L_BI}.

The rest of the paper is organized as follows. Next section is introductory to the case of under-determined systems. In Section~\ref{sec:preliminaries} we remind few notions and results.
Next we prove simple solvability conditions for \eqref{eq:main-nonlinear-P}. 
In main Section~\ref{sec:main-algorithms} we
propose few variants of general Newton algorithm \eqref{eq:newton-with-alpha_k-for-P}, including adaptive ones and estimate their convergence rate. Some particular
cases (scalar equations and inequalities, quadratic equations, problems with special structure) are treated
in Section~\ref{sec:special-cases}. Results of numerical simulation are exhibited in Section~\ref{sec:numerical}. Conclusion part finalizes the paper (Section~\ref{sec:conclusion}).

\section{Under-determined systems of equations}
\label{sec:under-determined}
Under-determined equations attracted our attention by specific norm-dependency property. In case of $m < n$ norms in $\R^n$ (in the optimization sub-problem) and $\R^m$ (for residual) can be chosen arbitrarily, and they imply principally different forms and results of Newton method \eqref{eq:newton-with-alpha_k-for-P}. Conditions on solvability and convergence look similar, but the results differ strongly.

Historically the case of under-determined equations ($m<n$) attracted much less attention than equations with the same number of equations and variables. 
The pioneering result is due to Graves \cite{Graves} in more general setting of Banach spaces, for problem \eqref{eq:main-nonlinear-y}. Graves' theorem for
finite-dimensional case claims, that if condition \[ ||g(x^a)-g(x^b)-A(x^a-x^b)|| \le C ||x^a-x^b|| \] holds in the
ball of radius $\rho$, centered at zero, for a matrix $A$ with minimal singular value $\mu > C > 0$, then a solution of the
equation \eqref{eq:main-nonlinear-y} exists provided $||y||$ is small enough, namely $ ||y|| \leq \rho (\mu -
C)$. The solution can be found via a version of modified Newton method, where next iteration requires solution of
the linear equation with matrix $A$, see \cite{Donchev,Magaril_Ilyaev-Tikhomirov-2008} for details.
The condition above gives rise to the mentioned metric regularity property.
However, finding the matrix $A$ is still a problem itself.

First of all let us specify the subproblem of finding a vector $z^k$ in \eqref{eq:newton-with-alpha_k-for-P}
for different norms of $x\in \R^n$. We skip simple verifications of the statements from convex analysis.
\begin{enumerate}
\item
For $||x||=||x||_1$ vector $z^k$ is a solution of the problem
\[
\min\{||z||_1 : P'(x^k)z=P(x^k) \}.
\]

\item
For $||x||=||x||_{\infty}$ vector $z^k$ is a solution of the problem
\[
\min\{ ||z||_{\infty} : P'(x^k)z = P(x^k) \}.
\]

Both problems above can be easily reduced to linear programming.

\item For $||x||=||x||_2$ vector $z^k$ can be written explicitly 
\[ z^k = P'(x^k)^\dagger P(x^k). \] 
In this case Newton method \eqref{eq:newton-with-alpha_k-for-P} coincide with 
\eqref{eq:newton-with-pseudoinverse}.
For $m \leq n$ Moore-Penrose pseudo-inverse of a matrix $A$ is written as $A^\dagger = A^T(A
A^T)^{-1}$, if $A$ has full row rank. 
\end{enumerate}

Thus in these (most important) cases algorithm
\eqref{eq:newton-with-alpha_k-for-P} can be implemented effectively. Also the solution of the first two problems may be non-unique.

An important case is the scalar one, i.e. $m=1$. We specify general results for
scalar equations and inequalities; the arising algorithms have much in common with unconstrained minimization methods. Finally we discuss nonlinear equations having some special structure. Then convergence results can be strongly enhanced.

\section{Preliminaries and feasibility (existence) theorems} \label{sec:preliminaries}

Key component in Newton method is the auxiliary convex optimization sub-problem, involving the linear constraint.
Note that the constraint 
\begin{equation} \label{eq:linear-equation} 
A z = b, \;\; b \in \R^m,\; z \in \R^n 
\end{equation} 
describes either a linear subspace, or the empty set. The classical
result below (which goes back to Banach, see \cite{Kant82,Magaril_Ilyaev-Tikhomirov-2008,Nesterov})
guarantees solvability of the linear equation (\ref{eq:linear-equation}) and gives an estimate of its solution. {\color{red}We prefer to provide the direct proof of the result because it is highly clear 
and short in finite-dimensional case.
Suppose} that spaces $\R^n, \R^m$ are equipped with some norms, the dual norms are denoted
$||\cdot||_*$ (for a linear functional $c$, associated with the vector of the same dimension,  
$\|c\|_* = \sup_{x : \|x\| = 1} (c, x)$).  
Operator {\color{red}norm is} subordinate with the vector norms, e.g. for $A: X\rightarrow Y$ we have
$\|Ax\|_Y \le \|A\|_{ X, Y} \|x\|_X$.
In most cases we do not specify vector norms; dual norms are obvious from the context. 
{\color{red}The adjoint operator $A^*$ is identified with matrix $A^T$.}

\begin{lemma} \label{lem:solution-of-linear-eq} If $A \in \R^{m \times n}$ satisfies condition 
\begin{equation}
\label{eq:mu-condition-for-A} \|A^T h\|_* \geq \mu_0 \|h\|_*, \;\; \mu_0 > 0, 
\end{equation}
for all $h \in \R^m$, then equation \eqref{eq:linear-equation} has a solution 
for all $b \in \R^m$, and all solutions of optimization problem 
\[ 
\widehat{z} \in \Arg \min \{\|z\| : A z = b\} 
\] 
have bounded norms $\|\widehat{z}\| \leq \frac{1}{\mu_0} {\color{red}\|b\|}$.
\end{lemma}

\begin{proof} 
{\color{red}
Fix $b \in \R^m$ and denote $K = \{x \in \R^n: ||x|| \leq \frac{\|b\|}{\mu_0}\}$.
This is a convex closed bounded set, 
and its linear image $Q = \{y \in \R^m: y = Ax, x \in K\}$ is convex closed bounded set as well. 
Suppose $b \notin Q$, then it can be strictly separated from $Q$: 
there exists $c \in \R^m: \max_{y \in Q}(c, y) < (c, b)$.
But $\max_{y \in Q}(c, y) = \max_{x \in K}(c, Ax) = \max_{x \in K}(A^Tc, x)
= \frac{\|b\|}{\mu_0}\|A^Tc\|_* \geq \|b\|\cdot\|c\|_*$, 
thus we get the contradiction: $\|b\| \cdot \|c\|_* < (c, b)$. 
Hence $b \in Q$, i.e. there exists $x \in \R^n: A x = b, \|x\| \leq \frac{||b||}{\mu_0}\}$.
A solution with the least norm obeys the same inequality.
}
\end{proof}

The Lemma is claiming that the matrix $A$ has full row rank equal to $m$ provided \eqref{eq:mu-condition-for-A} holds.
 It is another way to say that the mapping $A : \R^n \rightarrow \R^m$ is \emph{onto} mapping, i.e. covering all
image space. In the case of Euclidean norms, parameter $\mu_0$ is the smallest singular value of the matrix
$\mu_0 = \sigma_m(A)$ (we denote singular values of a matrix in $\R^{m \times n}$ in decreasing order as
$\sigma_1 \geq \sigma_2 \geq \ldots \geq \sigma_ m$). In general case the conjugate operator $A^*$ is used instead of $A^T$, and $\mu$ is the metric regularity constant.

Some results below will exploit the sum of double exponentials functions
$H_k \colon [0, 1) \rightarrow \R_+$, cf.~\cite{Polyak-1964}: 
$$H_k(\delta) = \sum_{\ell=k}^\infty \delta^{(2^\ell)}.$$ 
All functions $H_k(\cdot)$ are monotonically increasing and strictly convex. We also use two specific constants
\[
c_1 = H_0\Big(\frac{1}{2}\Big) \approx 0.8164215,
\]
and
\begin{equation} \label{eq:c_2-define}
c_2 = \max_{0 \leq r \leq \frac{1}{4}} 2 H_0\left(\frac{1}{2} - r\right)
 + 5 r - 4 r^2 - 2 c_1\approx 0.0036003.
\end{equation}
Trivial approximations
\begin{equation} \label{eq:H_k-Delta-approximations}
0 \leq H_k(\delta) \leq \frac{\delta^{(2^k)}}{1 - \delta^{(2^k)}}
= \frac{1}{\delta^{-(2^k)} - 1}.
\end{equation}
may be used for polynomial lower and upper bounds of $H_0(\delta) = \delta + \delta^2 + \delta^4 + \ldots + \delta^{(2^{k-1})} + H_k(\delta)$ with arbitrary precision. We also use property $H_k(\delta^2) = H_{k+1}(\delta)$. 

Below the problem of solvability of equation \eqref{eq:main-nonlinear-y} is addressed. We apply algorithm
\eqref{eq:newton-with-alpha_k-for-P} in the form
\begin{equation} \label{eq:newton-with-alpha_k-for-g} 
\begin{array}{l}
x^{k+1} = x^k - \alpha_k z^k, \;\; k = 0, 1, \ldots \\
z^k \in \Arg \min_z \{||z||: g'(x^k)z = g(x^k)-y\}.
\end{array}
\end{equation}
with small $\alpha$ and prove that the iterations converge while the limit
point is a solution. This techniques follows the idea from \cite{Polyak-1964}. Remind that $\R^n, \R^m$
are equipped with some norms, the dual norms are denoted $||\cdot||_*$.

\emph{Assumptions with respect to \eqref{eq:main-nonlinear-y}.}

$\mathbf{A}.$ 
$g(0)=0$ (i.e $\ol{x}=0$), $g(x)$ is differentiable in the ball $B=\{x\in \R^n: \|x\|\le \rho\}$, and its derivative $g'(x)$ satisfies Lipschitz condition in $B$:
\begin{equation} 
\nonumber \label{eq:lipschitz-deriv-g} 
\|g'(x^a) - g'(x^b)\| \leq L \|x^a - x^b\|. 
\end{equation}

$\mathbf{B}.$ 
The following inequality holds for all $x\in B$ and some fixed $\mu>0$: 
\begin{equation} \nonumber
 \|g'(x)^T h\|_* \geq \mu \|h\|_*, \; \; \forall h \in \R^m. 
 \end{equation}

$\mathbf{C}.$ 
$\|y\|<\mu\rho$.

\begin{theorem}\label{thm:existence-solution-for-g} 
If conditions $\mathbf{A}, \mathbf{B}, \mathbf{C}$ hold then there exists a solution $x^*$ of \eqref{eq:main-nonlinear-y}, and $\| x^*\| \leq \frac{\|y\|}{\mu}$. 
\end{theorem}

\begin{proof} 
We apply algorithm \eqref{eq:newton-with-alpha_k-for-g} with $\alpha>0$ small enough and
$x^0=0$. The algorithm is well defined --- condition $\mathbf{B}$ and Lemma~\ref{lem:solution-of-linear-eq} imply existence of solutions $z^k$ provided that $x^k\in B$; this is true for $k=0$ and will be validated recurrently. Standard formula
$$
g(x+z)=g(x)+\int_0^1 g'(x+tz)z dt
$$ 
combined with condition $\mathbf{A}$ provides for 
$x=x^k, \; z=-\alpha z^k$ and $u_k=\|g(x^k)-y\|$ recurrent relation 
$$
u_{k+1}\leq |1-\alpha| u_k+\frac{L\alpha^2}{2}\|z^k\|^2.
$$
Now condition $\mathbf{B}$ and Lemma~\ref{lem:solution-of-linear-eq} transform this estimate into 
$$u_{k+1}\leq |1-\alpha| u_k+\frac{L\alpha^2 u_k^2}{2\mu^2}. $$

Choose $\alpha = \eps \frac{2\mu^2}{L u_0}(1-\frac{u_0}{\mu\rho})$ with small $\eps < 1$ satisfying $0 < \alpha < 1$; it is possible due to condition $\mathbf{C}$. From the above inequality we get $u_{k+1} \leq u_k(1-\alpha + \alpha \eps \frac{u_k}{u_0}(1 - \frac{u_0}{\mu\rho}))$. 
For $k=0$ this implies $u_1 < u_0$ and recurrently $u_{k+1} < u_k$. We also get $u_{k+1} \leq q u_k,\; q = (1-\alpha + \alpha \eps (1-\frac{u_0}{\mu\rho})) < 1$. 
Thus $u_k \le q^k u_0$ and $u_k \rightarrow 0$ for $k\rightarrow \infty$.

On the other hand we have $\|x^{k+1}-x^k\|=\alpha\|z^k\| \le \frac{\|g(x^k)-y\|}{\mu} = \frac{u_k}{\mu} \le q^k \frac{u_0}{\mu}$. Hence for any $k, s$ and for
$k\rightarrow\infty$ 
\[ 
\|x^{k+s} - x^k\| \leq \sum_{i=k}^{k+s-1}\|x^{i+1}-x^i\| \leq q^k \frac{u_0}{(1-q)\mu} \rightarrow 0. 
\]
It means that $x^k$ is a Cauchy sequence and converges to some point $x^*(\eps)$. We
had $g(x^k) \rightarrow y$, thus continuity reasons imply $g(x^*(\eps))=y$. Now, for all iterations we got $\|x^k-x^0\|=\|x^k\| \leq \sum_{j=0}^{k-1} \|x^{j+1} - x^j\|\leq \alpha
\frac{u_0}{\mu (1-q)} = \frac{u_0}{\mu}\frac{1}{1-\eps (1-\frac{u_0}{\mu\rho})} < \frac{u_0}{\mu}\frac{1}{1- (1-\frac{u_0}{\mu\rho})} =
\rho$. 
Hence all iterations $x^k$ remain in the ball $B$ and our reasoning was correct.
Finally under $\|x^k\| \leq \frac{u_0}{\mu} \frac{1}{1 - \varepsilon(1 - \frac{u_0}{\mu \rho})}$ we take $\varepsilon\rightarrow 0$, leading to $\|x^k\| \leq \frac{u_0}{\mu}$, so its limit point $x^*(\varepsilon)|_{\varepsilon \rightarrow 0}$. The limit point $x^*(\eps)|_{\eps \rightarrow 0} = x^*$ is a solution as well and $\|x^*\|\le \frac{u_0}{\mu}$.
\end{proof}

\begin{corollary} \label{cor:global-convergence-for-rho-infty} 
If $\rho=\infty$ (that is conditions $\mathbf{A}, \mathbf{B}$ hold on the entire space $\R^n$) then equation \eqref{eq:main-nonlinear-y} has a solution for an arbitrary right-hand side $y$.
\end{corollary}

It is worth noting that if we apply pure Newton method (i.e. take $\alpha_k\equiv 1$), the conditions of its convergence are more restrictive: we need $\|y\|\le \frac{2\mu^2}{L}$, that is we guarantee only local convergence even for $\rho=\infty$. This is a corollary of Newton-Mysovskikh theorem \cite{Kant82}, which proof is valid for under-determined case as well, cf. also \cite{Polyak-1964}.

\begin{corollary}
If $m=n$ and Condition $\mathbf{B}$ is replaced with $\|g'(x)^{-1}\|\leq \frac{1}{\mu}, x\in B$, then the statement of Theorem~\ref{thm:existence-solution-for-g} holds true.
\end{corollary}
In this case our method \eqref{eq:newton-with-alpha_k-for-g} reduces to classical Newton method \eqref{eq:classic-newton}.

There exist numerous results on solvability of \eqref{eq:main-nonlinear-y}. Some of them are stronger than Theorem~\ref{thm:existence-solution-for-g} and are based on general notion of metric regularity \cite{Donchev-Rock,Ioffe}. We provided the proof based on our technique to exhibit its applicability to existence theorems.

\section{Main algorithms}\label{sec:main-algorithms}
Here it is more convenient to use the main equation in form \eqref{eq:main-nonlinear-P}. 
In previous Section we proved solvability of equation by use of the algorithm with constant $\alpha_k \equiv \alpha>0$; choosing $\alpha$ smaller we obtained larger solvability domain. 

However in this Section our goal is different --- to reach the fastest convergence to a solution. For this purpose different strategies for
design of step-sizes are needed. The basic policy is as follows. First, we rewrite assumptions in new notation. We remark that the assumptions in context of equation $P(x) = 0$ are very much the same as $\mathbf{A}, \mathbf{B}$.

$\mathbf{A'}.$ $P(x)$ is differentiable in the ball $B=\{x\in \R^n: \|x-x^0\|\le \rho\}$, and its derivative $P'(x)$ satisfies
Lipschitz condition in $B$:
\begin{equation} \nonumber \label{eq:lipschitz-deriv-P} 
\|P'(x^a) - P'(x^b)\| \leq L \|x^a - x^b\|. 
\end{equation}

$\mathbf{B'}.$ The following inequality holds for all $x\in B$ and some $\mu>0$: 
\begin{equation} \nonumber
 \|P'(x)^T h\|_* \geq \mu
\|h\|_*, \; \; \forall h \in \R^m. 
\end{equation}

If $\mathbf{A'}, \mathbf{B'}$ hold true, we have the same recurrent inequalities for $u_k=\|P({\color{red}x^k})\|$:
\begin{equation}
\label{eq:ineq1}
u_{k+1}\le |1-\alpha_k| u_k + \frac{L\alpha_k^2\|z^k\|^2}{2},
\end{equation}

\begin{equation}
\label{eq:ineq2}
u_{k+1}\le |1-\alpha_k| u_k + \frac{L\alpha_k^2u_k^2}{2\mu^2},
\end{equation}
the second one being just continuation of the first one based on the estimate $\|z^k\| \le \frac{u_k}{\mu}$, compare with the calculations in the proof of Theorem~\ref{thm:existence-solution-for-g}. Now we can minimize right-hand sides of these inequalities over $\alpha_k$; it is natural to expect that such choice of step-size imply the fastest convergence of $u_k$ to zero and thus the fastest convergence of iterations $x_k$ to the solution. One of the main contributions of this paper is careful analysis of the resulting method.

If one applies such policy based on inequality \eqref{eq:ineq2}, optimal $\alpha$ depends on $\mu, L$ (Algorithm~1 below). The values are hard to estimate in most applications, thus the method would be hard for implementation. Fortunately, we can modify the algorithm using parameter adjustment (Algorithm~2). On the other hand the same policy
based on \eqref{eq:ineq1} requires just the value $L$, which is commonly available (Algorithm~3).

Thus we arrive to an algorithm which we call \emph{Newton method} while in fact it is blended \emph{pure Newton} with \emph{damped Newton} with special rule for damping. 
In some relation it reminds \emph{Newton method} for minimization of self-concordant functions \cite{NN}. Despite its simplicity, the idea of minimizing upper bound seems to be unexplored (or long forgotten) in Newton method theory with respect to equations. Authors found similar choice of step-size in \cite{Deuflhard-2004}, in different conditions and without explicit convergence bounds.

\subsection{Newton method with known constants}
\label{sec:newton-L-mu-known}

If both constants $L$ and $\mu$ are known, then the step-size is taken as the minimizer of right-hand side of \eqref{eq:ineq2}:
\begin{equation} \label{eq:alpha_k-for-L-mu}
\alpha_k = \arg \min_\alpha \Big( |1-\alpha|\cdot \|P(x^k)\| + \frac{L\alpha^2 \|P(x^k)\|^2}{2\mu^2} \Big) = \min\Big\{1, \frac{\mu^2}{L \|P(x^k)\|} \Big\}.
\end{equation}

\begin{framed}
    \begin{minipage}{0.8\linewidth}
        Algorithm 1 (Basic Newton method)
        \begin{equation} \nonumber
        \begin{aligned}
            & z^k \in \Arg \min_{P'(x^k) z = P(x^k)} \|z\|,
        \end{aligned}
        \end{equation}
        \begin{equation} \label{eq:x_k-update-algorithm1}
        \begin{aligned}
            & x^{k+1} = x^k - \min\Big\{1, \frac{\mu^2}{L \|P(x^k)\|} \Big\}z^k, \; k \geq 0.
        \end{aligned}
        \end{equation}
    \end{minipage}
\end{framed}

The algorithm is well-defined, as soon $\|P(x^k)\| = 0$ means that a solution is already found (formally $z^k = 0, \; \alpha_k = 0$ thereafter).
We remind that in calculation of $z^k$ any vector norm in $\R^n$ can be used, also any vector norm in $\R^m$ can be used for $\|P(x^k)\|$, and constants $L, \mu$ must comply with these norms.

The update step in \eqref{eq:x_k-update-algorithm1} can be written in less compact but more illustrative form:
\[
\begin{array}{lll}
\displaystyle
x^{k+1} = x^k - \frac{\mu^2}{L \|P(x^k)\|} z^k, & \displaystyle \text{ if } \|P(x^k)\| \geq \frac{\mu^2}{L} & \text{ (Stage 1 step)}, \\[4mm]
\displaystyle x^{k+1} = x^k - z^k, & \displaystyle \text{ otherwise} &\text{ (Stage 2 step)}. \\
\end{array}
\]
The latter case is a pure Newton step while the primal one is a damped Newton step. Direction $z^k$ calculation is the same in both stages. The result on convergence and rate of convergence is given below.
We use upper ($\lceil \cdot \rceil$) and lower ($\lfloor \cdot \rfloor$) rounding to integer; constant $c_1 \approx 0.8164$ was introduced in the end of Section~\ref{sec:preliminaries}.
The theorem is followed by the corollary with simpler statements.

\begin{theorem} \label{thm:alg1-exact} Suppose that Assumptions~$\mathbf{A'}, \mathbf{B'}$ hold and
\begin{equation} \label{eq:cond-P_0-alg1}
\|P(x^0)\| \leq \frac{\mu^2}{L} F^{\inv}_1\Big( \frac{L}{\mu} \rho \Big), 
\end{equation} 
where $F^{\inv}_1(\cdot)$ is the inverse function for the continuous strictly increasing function $F_1(w)$, given by
\begin{subequations} \label{eq:cond-p-alg1}
\begin{empheq}[left={\displaystyle F_1(w) = \empheqlbrace}]{align}
\label{eq:cond-p-alg1-a}
\displaystyle
&2 H_0\Big(\frac{w}{2}\Big), & 0 \leq w \leq 1, \\
\label{eq:cond-p-alg1-b}
&\lceil 2 w \rceil - 2 + 2 H_0\Big(\frac{1}{2} - \frac{\lceil 2 w \rceil - 2 w}{4}\Big), & w > 1.
\end{empheq}
\end{subequations}

Then the sequence $\{x^k\}$ generated by Algorithm 1 converges to a solution $x^* : P(x^*) = 0$.
The values $\|P(x^k)\|$ are monotonically decreasing, and there are not more than 
\begin{equation} \label{eq:k_max-definition} 
k_{\max} = \max\Big\{0, \;\left\lceil \frac{2 L}{\mu^2}\|P(x^0)\|\right\rceil - 2 \Big\} 
\end{equation}
iterations at Stage~1, then followed by Stage~2 steps. At $k$-th step the following estimates for the rate of convergence hold:
\begin{subequations} \label{eq:P_k-and-x_k-alg1-exact} 
\begin{align} 
\label{eq:P_k-alg1-stage1-exact} 
\|P(x^k)\| & \leq \|P(x^0)\| - \frac{\mu^2}{2L}k, & k < k_{\max}, 
\\ 
\label{eq:x_k-alg1-stage1-exact}
\|x^{k} - x^*\| & \leq \frac{\mu}{L} \big( k_{\max} - k + 2 H_0\Big(\frac{\ol{w}}{2}\Big) \big), & k < k_{\max},
\\ 
\label{eq:P_k-alg1-stage2-exact}
\|P(x^k)\| & \leq \frac{2\mu^2}{L} \Big(\frac{\ol{w}}{2}\Big)^{(2^{(k-k_{\max})})}, & k \geq k_{\max},
\\ 
\label{eq:x_k-alg1-stage2-exact}
\|x^{k} - x^*\| & \leq \frac{2\mu}{L}H_{k - k_{\max}}\Big(\frac{\ol{w}}{2}\Big), & k \geq k_{\max}.
\end{align}
\end{subequations} 
where $\ol{w} = \frac{L}{\mu^2}\|P(x^0)\| - \frac{k_{\max}}{2} = \min\big\{\frac{L}{\mu^2}\|P(x^0)\|,$ $1 - \frac{1}{2}\big\lceil \frac{2L}{\mu^2}\|P(x^0)\|\big\rceil + \frac{L}{\mu^2}\|P(x^0)\| \big\} \in [0, 1)$.
\end{theorem}
The Theorem's statement may look quite involved, but both functions $H_0(\delta)$ and $F_1^{\inv}(p)$ are easily calculated in practice. In the interval of interest $\delta \in [0, \frac{1}{2}]$, the former function has rational approximation \eqref{eq:H_k-Delta-approximations}. The latter function can be evaluated with needed accuracy via binary search, as soon $F_1(w)$ is monotonically increasing on $R_+$.

\begin{proof}
Assume that all $x^k \in B, \; k \geq 0$. Below we state condition enabling this assumption.
Using $w_k = \frac{L}{\mu^2}\|P(x^k)\|$ as the objective function, we rewrite \eqref{eq:ineq2} with generic step-size $\alpha$ as
\begin{equation} \label{eq:w_k-alg1-relation}
w_{k+1} \leq |1 - \alpha| w_k + \frac{1}{2}\alpha^2 w_k^2.
\end{equation}
Its optimum over $\alpha$ is at $\alpha_k = \frac{1}{w_k} < 1$,
if $w_k > 1$; and $\alpha_k = 1$ otherwise; it is exactly \eqref{eq:alpha_k-for-L-mu}.

During Stage~1 of damped Newton steps ($\alpha_k < 1$), the objective function monotonically decreases as
\begin{equation} \label{eq:w_k-alg1-stage1}
w_{k+1} \leq w_k - \frac{1}{2}.
\end{equation}
There are at most $k_{\max} = \max\{0, \; \lceil 2 w_0 \rceil - 2\}$ iterations in the phase, say $\ol{k}$ ones, resulting in $w_{\ol{k}} \leq 1$. As soon $w_k$ reaches this unit threshold, the algorithm switches to Stage~2, pure Newton steps. Then recurrent relation \eqref{eq:w_k-alg1-relation} becomes 
\[
w_{k+1} \leq \frac{1}{2} w_k^2, \; k \geq \ol{k}. 
\] 
so we can write
\begin{equation} \label{eq:w_k-alg1-stage2} 
w_{\ol{k}+\ell} \leq 2 \left(\frac{w_{\ol{k}}}{2}\right)^{(2^\ell)}
, \; \ell \geq 0. 
\end{equation}

For the second phase $\|x^{i+1} - x^i\| = \|z^i\| \leq \frac{1}{\mu} \|P(x^i)\| = \frac{\mu}{L} w_i$ due Lemma~\ref{lem:solution-of-linear-eq}, and for $\ell_2 \geq \ell_1 \geq 0$ holds 
\begin{equation} \label{eq:x_a-x_b-alg1-stage2} 
\|x^{\ol{k} + \ell_2} - x^{\ol{k} + \ell_1}\| 
\leq \sum_{i = \ell_1}^{\ell_2-1} \|x^{\ol{k}+i+1} - x^{\ol{k}+i}\| 
\leq \frac{2 \mu}{L} \big(H_{\ell_1}\Big(\frac{w_{\ol{k}}}{2}\Big) -
H_{\ell_2}\Big(\frac{w_{\ol{k}}}{2}\Big)\big). 
\end{equation} 
The sequence $\{x^k\}$ is a Cauchy sequence because $H_j(\frac{w_{\ol{k}}}{2}) \leq H_j(\frac12) \rightarrow_{j \rightarrow \infty} 0$. It converges to a point $x^* : \|P(x^*)\| = \lim_{k \rightarrow \infty} \|P(x^k)\| = 0$ due to continuity of $P$, with 
\begin{equation} \label{eq:x_k-x*-alg1-stage2} 
\|x^{\ol{k} + \ell} - x^*\| 
\leq \frac{2 \mu}{L} H_\ell \left(\frac{w_{\ol{k}}}{2}\right) 
, \; \ell \geq 0. 
\end{equation}

Next we are to estimate distance from points $x^k$ in Stage 1 to the limit solution point $x^*$. One-step distance for $k < \ol{k}$ is bounded by a constant: $\|x^{k+1} - x^k\| = \alpha_k\|z^k\| \leq \alpha_k {\color{red}\frac{\mu}{L}} w_k = \frac{\mu}{L},$ and altogether
\begin{equation} \label{eq:x_k-x*-alg1} 
\|x^{k} - x^*\| 
\leq \|x^{\ol{k}} - x^*\| + \sum_{i = k}^{\ol{k}-1} \|x^{i+1} - x^i\| 
\leq \frac{\mu}{L} \big(\ol{k} - k + 2 H_0\left(\frac{w_{\ol{k}}}{2}\right) \big), \; k < \ol{k}. 
\end{equation} 
Note that the formula also coincides with the upper bound \eqref{eq:x_k-x*-alg1-stage2} at $k = \ol{k}$. Exact number $\ol{k}$ of the steps in the first phase is not known, but we can replace it with the upper bound $k_{\max}$ in all estimates \eqref{eq:w_k-alg1-stage1}--\eqref{eq:x_k-x*-alg1}, due to monotonic decrease of $\{w_k\}$. We also have an upper bound for $w_{k_{\max}} \leq \ol{w} = w_0 - \frac{1}{2} k_{\max} = w_0 - \frac{\max\{0, \lceil 2 w_0 \rceil - 2\}}{2} \in [0, 1]$.
Substituting $w_k = \frac{L}{\mu^2}\|P(x^k)\|$ back we arrive to Theorem~\ref{thm:alg1-exact} bounds \eqref{eq:P_k-and-x_k-alg1-exact}.

Finally we are to check our primal assumption of the algorithm-generated points $x^k$ being within $B$. This is guaranteed by one of two conditions, depending on whether the Algorithm starts from Stage~1 step or Stage~2 step.

In the first case $w_0 > 1$, and $\|x^0 - x^k\|$ is bounded similarly to~\eqref{eq:x_k-x*-alg1} as
\begin{align}
\nonumber
\|x^0 - x^k\| \leq & \sum_{i = 0}^{k-1} \|x^{i+1} - x^i\| \leq \sum_{i = 0}^{\ol{k}-1} \|x^{i+1} - x^i\| + \sum_{i = \ol{k}}^{\infty} \|x^{i+1} - x^i\|\leq \\
\nonumber 
\leq & \frac{\mu}{L} \big(\ol{k} + 2 H_0\left(\frac{w_{\ol{k}}}{2}\right) \big) 
\leq \frac{\mu}{L} \big(k_{\max} + 2 H_0\left(\frac{w_{k_{\max}}}{2}\right)\big) \leq \\
\nonumber
\leq & \frac{\mu}{L} \big(k_{\max} + 2 H_0\left(\frac{\ol{w}}{2}\right)\big) = \\
= & \frac{\mu}{L} \left( \lceil 2 w_0 \rceil - 2 + 2 H_0\Big(\frac{1}{2} - \frac{\lceil 2 w_0 \rceil - 2 w_0}{4}\Big) \right).
\end{align}
Here we also used $k_{\max} = \lceil 2 w_0 \rceil - 2 > 0$ and the upper bound $w_{k_{\max}} \leq \ol{w}$. In other words, given $w_0 = \frac{L}{\mu^2}\|P(x^0)\| > 1$, for $\{x^k\} \in B$ it is sufficient to satisfy $\frac{L}{\mu} \rho \geq F_1\big(w_0\big)$. This corresponds to \eqref{eq:cond-p-alg1-b} part.

In the second case we have $w_0 \leq 1$, and the algorithm makes pure Newton steps with $\alpha_k \equiv 1$ from the beginning. Then $\ol{k} = 0$, $w_{k_{\max}} = w_0$ and from \eqref{eq:x_a-x_b-alg1-stage2} follows
\[
\|x^k - x^0\| \leq \frac{2 \mu}{L} (H_0\Big(\frac{w_0}{2}\Big) - H_k\Big(\frac{w_0}{2}\Big)) \leq \frac{2 \mu}{L} H_0\Big(\frac{w_0}{2}\Big), \; k \geq 0.
\]
Therefore if $w_0 \leq 1$, then inequality $\|x^0 - x^k\| \leq \rho, \; k \geq 0$ is satisfied if $\frac{L}{\mu} \rho \geq 2 H_0\big(\frac{w_0}{2}\big) {\color{red} = F_1(w_0)}$.
This corresponds to \eqref{eq:cond-p-alg1-a} part.

Gluing the cases $w_0 \leq 1$ and $w_0 > 1$ we arrive to the sufficient condition $\rho \geq \frac{\mu}{L}F_1\big({\color{red}\frac{L}{\mu^2}} \|P(x^0)\| \big)$, resulting in $x^k \in B$. Due to $F_1(w)$ being strictly increasing this condition is equivalent to \eqref{eq:cond-P_0-alg1}.
\end{proof}

Result on the rate of convergence means, roughly speaking, that after no more than $k_{\max}$ iterations one has very fast (quadratic) convergence. For good initial approximations $k_{\max}=0$, and pure Newton method steps are performed from the very start.

\begin{corollary}
If $\rho = \infty$ (that is conditions $\mathbf{A'}, \mathbf{B'}$ hold on the entire space $\R^n$), 
then Algorithm~1 converges to a solution of \eqref{eq:main-nonlinear-P} for any $x^0 \in \R^n$.
\end{corollary}

The following corollary provides simpler tight {\color{red}relaxed} condition for Theorem~\ref{thm:alg1-exact}. 
The idea is to develop an upper bound for \eqref{eq:cond-p-alg1}, resulting in a  lower bound on \eqref{eq:cond-P_0-alg1}.
\begin{corollary} \label{cor:alg1-simple-bound}
Condition \eqref{eq:cond-P_0-alg1} can be replaced with piece-wise linear one
\[
\|P(x^0)\| \leq \frac{\mu^2}{L}F_1^{\inv, \lw}\Big( \frac{L}{\mu}\rho \Big) = \frac{\mu^2}{L} \times \left\{
\begin{array}{cl}
\displaystyle
\frac{1}{2 (2 c_1 - 1)} \frac{L}{\mu}\rho, & \displaystyle \;\; 0 \leq \rho \leq (2 c_1 - 1)\frac{\mu}{L}, \\[4mm]
\displaystyle
\frac{L}{2\mu}\rho + 1 - c_1, & \displaystyle \;\; \rho > (2 c_1 - 1)\frac{\mu}{L}.
\end{array}
\right.
\]
\end{corollary}
\begin{proof}
In order to find \emph{a lower} bound for $F_1^{\inv}(\cdot)$, we are to prove \emph{an upper} bound for the function $F_1(w)$, \eqref{eq:cond-p-alg1}. Both bounds are continuous strictly increasing functions.
First we notice, that both \eqref{eq:cond-p-alg1-a} and \eqref{eq:cond-p-alg1-b}
coincide at $w \in [\frac{1}{2}, 1]$, and it can be rewritten through the different junction point $w = 1/2$ instead of $w = 1$ 
\begin{equation} \nonumber
F_1(w) = 
\left\{
\begin{array}{ll}
\displaystyle
2 H_0\Big(\frac{w}{2}\Big), & 
\displaystyle 0 \leq w \leq \frac{1}{2}, \\
\displaystyle
\lceil 2 w \rceil - 2 + 2 H_0\Big(\frac{1}{2} - \frac{\lceil 2 w \rceil - 2 w}{4}\Big), & 
\displaystyle w > \frac{1}{2}.
\end{array}
\right.
\end{equation}
Due to convexity on interval $[0, \frac{1}{2}]$, 
function $H_0(\delta)$ is bounded by a secant segment:
$$
2 H_0\Big(\frac{w}{2}\Big) \leq 2 H_0\Big(\frac{1}{4}\Big)\cdot(2 w) = 2(2 c_1 - 1) w.
$$
Here we used property $H_0\big(\frac{1}{4}\big) = H_0\big(\frac{1}{2}\big) - \frac{1}{2}$, which follows from the {\color{red} identity} $H_k(x) = x^{(2^k)} + H_k(x^2), \;\; x {\color{red}\in[0, 1)}$; and constant $c_1 = H_0(\frac{1}{2})$ is introduced in Section~\ref{sec:preliminaries}.

Next we show that $F_1(w) \leq 2 (w + c_1 - 1)$ {\color{red} for} $w \geq \frac{1}{2}$.
Indeed, following continuous function is periodic on $w, \; w \geq \frac{1}{2}$. It can be also written through the variable $r = \frac{1 + 2 w - \lceil 2w \rceil}{2} \in (0, \frac{1}{2}]$.
\begin{align}
\nonumber 
F_1(w)\! -\! 2(w + c_1\! -\! 1)\! & =\! 2\left(\! \frac{\lceil 2w \rceil\! -\! 2 w}{2} \!+\! H_0\Big(\frac{1}{2} \!-\! \frac{\lceil 2w \rceil \!-\! 2w}{4} \Big) \!-\! H_0\Big(\frac{1}{2}\Big) \!\right) = \\
\nonumber
& = 2\left(-r + H_0\left(\frac{1}{4} + \frac{r}{2} \right) - H_0\left({\color{red}\frac{1}{4}}\right) \right) =\\
& = 2\left(H_0\left(\frac{1}{4} + \frac{r}{2} \right) - \Big(H_0\left(\frac{1}{4}\right) + r\Big)\right) \leq 0.
\end{align}
{\color{red}In the second row identity $H_0(\frac{1}{2})=\frac{1}{2}+H_0(\frac{1}{4})$ is used.
The last inequality is due to convexity of $H_0(\delta)$, which is under secant segment on corresponding interval $[\frac{1}{4}, \frac{1}{2}]$.}

Thus we have the monotonically increasing piece-wise linear upper bound for \eqref{eq:cond-p-alg1}
\[
F_1(w) \leq F_1^{\up}(w) = \left\{
\begin{array}{cl}
\displaystyle
2 (2 c_1 - 1) w, & \displaystyle \; 0 \leq w \leq \frac{1}{2}, \\[2mm]
\displaystyle
2 ( w + c_1 - 1), & \displaystyle \; w > \frac{1}{2}.
\end{array}
\right.
\]
Its inverse {\color{red}$F_1^{\inv, \lw}$ with property} $F_1^{\inv, \lw}(F_1^{{\color{red}\up}}(w)) \equiv w, w \geq 0$ is the piece-wise linear lower bound for $F_1^{\inv}$:
\[
F_1^{\inv}(p) \geq F_1^{\inv, \lw}(p) = \left\{
\begin{array}{cl}
\displaystyle
\frac{1}{2 (2 c_1 - 1)} p, & \displaystyle \; 0 \leq p \leq 2 c_1 - 1, \\[4mm]
\displaystyle
\frac{p}{2} + 1 - c_1, & \displaystyle \; p > 2 c_1 - 1.
\end{array}
\right.
\]
Substituting back $p = \frac{L}{\mu}\rho$ and $w = \frac{L}{\mu^2} \|P(x^0)\|$ results in the Corollary statement.
\end{proof}
On Figure~\ref{fig:alg1-w-on-p-bounds} the bound and its residuals are plotted. From the proof it is clear the bounds are tight.

The linear upper (lower) bounds of $F_1(\cdot), F_1^{\inv}(\cdot)$ for the interval $w \in [0, \frac{1}{2}]$ were chosen for consistency with linear bounds on $w \in [\frac{1}{2}, \infty)$. Bound residual on these two intervals are also in the same order, cf. Figure~\ref{fig:alg1-w-on-p-bounds}. In Corollary~\ref{cor:alg3-simple-bound} of Theorem~\ref{thm:alg3-exact} we present refined, quadratic approximation for $F_1(w) = 2H_0(\frac{w}{2}), \; w \leq \frac{1}{2}$.

\begin{figure} 
\caption{Algorithm 1: function $F_1^{\inv}(p)$ and lower bound $F_1^{\inv,\lw}(p)$ (left), residual $F_1^{\inv}(p) - F_1^{\inv, \lw}(p)$ (right).}
\label{fig:alg1-w-on-p-bounds}
\includegraphics[width=0.5\linewidth]{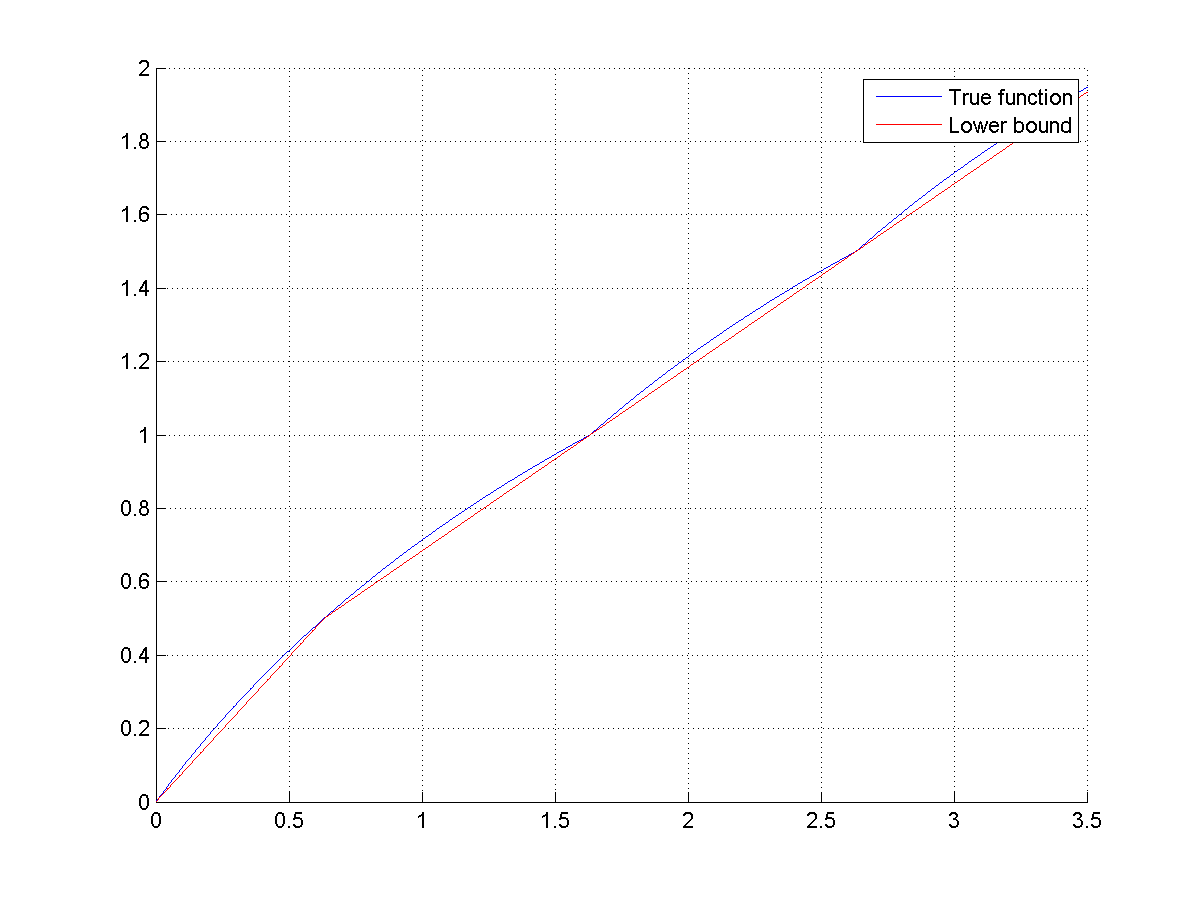}%
\label{fig:alg1-w-on-p-residual}
\includegraphics[width=0.5\linewidth]{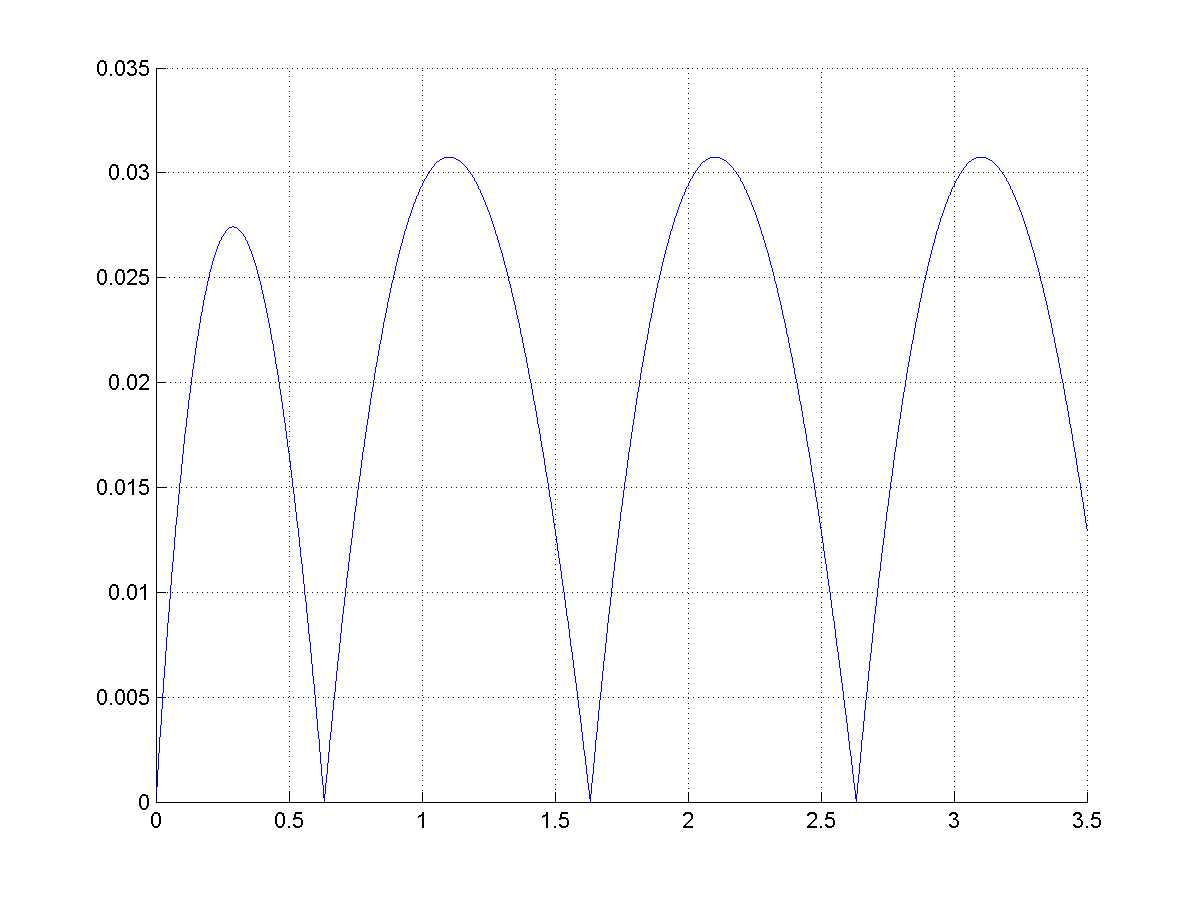}
\end{figure}

\begin{corollary}
The upper bounds \eqref{eq:P_k-and-x_k-alg1-exact} may be simplified as well, using $\ol{w} \leq 1$ and thus $H_0(\frac{\ol{w}}{2}) \leq H_0(\frac{1}{2})= c_1$:
\begin{subequations} \nonumber \label{eq:P_k-and-x_k-alg1} 
\begin{align} 
\|x^{k} - x^*\| & \leq \frac{\mu}{L} ( k_{\max} - k + 2 c_1), & k < k_{\max},
\\ 
\|P(x^k)\| & \leq \frac{\mu^2}{L} \frac{1}{2^{(2^{k-k_{\max}}) - 1}}, & k \geq k_{\max},
\\ 
\|x^{k} - x^*\| & \leq \frac{2\mu}{L}H_{k - k_{\max}}\Big(\frac{1}{2}\Big) \leq \frac{2\mu}{L} \frac{1}{2^{(2^{k - k_{\max}})} - 1}, & k \geq k_{\max}.
\end{align}
\end{subequations}
\end{corollary}

\subsection{Adaptive Newton method}
Presented Algorithm~1 explicitly uses two constants $\mu$ and $L$ but both enter into the algorithm as one parameter $\beta = \frac{\mu^2}{L}$. There is a simple modification allowing adaptively change an estimate of the parameter.

Input of the algorithm is an initial point $x^0$, the approximation $\beta_0$ and the parameter $0 < q < 1$.
\begin{framed} \begin{minipage}{0.8\linewidth}
Algorithm 2 (Adaptive Newton method) 
\begin{enumerate}
\item
Calculate
\begin{equation} \nonumber
\begin{array}{ll}
 & \displaystyle z^k \in \Arg \min_{P'(x^k) z = P(x^k)} \|z\|,\\
 & \alpha_k = \min\big\{1, \frac{\beta_k}{\|P(x^k)\|}\big\}, \\
 & u_{k+1} = \|P(x^k - \alpha_k z^k)\|.
\end{array}
\end{equation}

\item
If either
\begin{equation} \nonumber
\alpha_k < 1 \text{ and } u_{k+1} < u_k - \frac{\beta_k}{2},
\end{equation}
or
\begin{equation} \nonumber
 \alpha_k = 1 \text{ and } u_{k+1} < \frac{1}{2 \beta_k} u_k^2,
 \end{equation} holds, go to Step 4.
Otherwise

\item
apply update rule $\beta_{k} \leftarrow q \beta_k$ and return to Step~1 without increasing the counter.

\item
Take
\[
x^{k+1} = x^k - \alpha_k z^k,
\]
set $\beta_{k+1} = \beta_k$, increase counter $k \leftarrow k+1$, and go to Step~1.
\end{enumerate}
\end{minipage}
\end{framed}

Properties of Algorithm~2 are similar to Algorithm~1. We omit the formal proof of convergence; it follows the lines of the proof of Theorem~\ref{thm:alg1-exact} with respect to the properties:
\begin{itemize}
\item Algorithm~2 runs real steps at Step~4 and some number of fictitious steps resulting in update rule Step~3;
\item
$\beta_k$ is non-increasing sequence;
\item
if $\beta_k < \beta$ (the actual constant of the objective function), then Step~3 won't appear and $\beta_k$ won't decrease anymore. It means that there are at most $\wh{k} = \max\{0, \; \lceil\log_{1/q} (\frac{\beta_0}{\beta})\rceil\}$ check steps. Minimal possible value of $\beta_k$ is $\beta_{\min} = q^{\wh{k}} {\color{red}\beta_0}$, and the number of Stage~1 steps is limited by $\wh{k}_{\max} = \max\{0, \; \left\lceil 2 \frac{\|P(x^0)\|}{\beta_{\min}}\right\rceil - 2\}$ as well;

\item
if Step~4 is performed with $\beta_k > \beta$ due to validity of a condition in Step~2, then $\|P(x^{k+1})\|$ decreases \emph{more} than at the corresponding step with ``optimal'' step-size $\alpha_k = \min\{1, \; \frac{\beta}{\|P(x^k)\|}\}$ (calculated with ``true'' value $\beta$).
\end{itemize}

Let us mention two other versions of adaptive Newton method. The first one uses increasing updates (e.g. $\beta_{k+1} = q_2 \beta_k$ with $q_2 > 1$) in the end of Step~4, thus adapting the constant to current $x^k$. Also other decrease policies can be applied for $\beta_k$ in Step~3.

The alternative to the Algorithm~2
is line-search or Armijo-like rules for choosing step-size $\alpha_k$ to minimize objective function $\|P(x^k - \alpha z^k)\|$ directly. 
It is known that this approach eventually leads to the quadratic convergence rate with pure Newton steps as well, but without any estimates \cite{Burdakov}. 
The difference between {\color{red} is} the following: in the proposed Algorithm~2 parameter $\beta$ is being \emph{monotonically tuned} to the global problem-specific value, 
rather than trial-and-error procedure is performed at \emph{every} iteration in the Armijo-like approach. 
We compare the alternatives in Example~1.

\subsection{Method for $L$ known}
\label{ssec:newton-L-known}
Constant $\mu$, used in Assumptions $\mathbf{B}, \mathbf{B'},$ is rarely accessible. 
As said in the beginning of the section, we can use more accurate upper bound \eqref{eq:ineq1} instead of \eqref{eq:ineq2} for step-size choice. It results in the algorithm, which uses the Lipschitz constant only. 
The optimal step-size in this case is
\begin{equation} \label{eq:alg3-alpha}
\alpha^*_k = \arg\min_{\alpha} \Big(|1-\alpha|\cdot \|P(x^k)\| + \frac{L\alpha^2\|z^k\|^2}{2} \Big) = \min\Big\{1, \frac{\|P(x^k)\|}{L \|z^k\|^2} \Big\}.
\end{equation}
\begin{framed}
   \begin{minipage}{0.8\linewidth}
        Algorithm 3 ($L$-Newton method)
        \begin{equation} \nonumber
        \begin{aligned}
            & z^k \in \Arg \min_{P'(x^k) z = P(x^k)} \|z\|,
        \end{aligned}
        \end{equation}
        \begin{equation} \nonumber \label{eq:x_k-update-algorithm3}
        \begin{aligned}
            & x^{k+1} = x^k - \min\Big\{1, \frac{\|P(x^k)\|}{L \|z^k\|^2} \Big\}z^k, \; k \geq 0.
        \end{aligned}
        \end{equation}
    \end{minipage}
\end{framed}
The algorithm is well-defined, as condition $\|z^k\| = 0$ holds only if $P(x^k) = 0$, i.e. a solution was found at the previous step. Formally we put $z^k = 0, \; \alpha_k = 1$ and $x^{k+1} = x^k$ thereafter.

For the Algorithm we also present similar convergence theorem and set of corollaries.
We emphasize that while constant $\mu$ is still used in the bounds, Algorithm~3 does not depend on it.
\begin{theorem} \label{thm:alg3-exact} Suppose that Assumptions~$\mathbf{A'}, \mathbf{B'}$ hold and
\begin{equation} \label{eq:cond-P_0-alg3}
\|P(x^0)\| \leq \frac{\mu^2}{L} F^{\inv}_2\Big( \frac{L}{\mu} \rho \Big), 
\end{equation} 
where $F^{\inv}_2(\cdot)$ is the inverse function for the continuous strictly increasing function $F_1(w)$ given by
\begin{subequations} \nonumber \label{eq:cond-p-alg3}
\begin{empheq}[left={\displaystyle \!\!\!F_2(w) \!=\!\! \empheqlbrace}]{align}
\label{eq:cond-p-alg3-a}
\displaystyle
& 2 H_0\Big(\frac{w}{2}\Big), \hspace{58mm} 0 \leq w \leq 1, \\
\label{eq:cond-p-alg3-b}
& \frac{(\lceil 2w \rceil \! - \! 2) (4 w \! - \!\lceil 2w \rceil \! +\! 3)}{4} 
\! +\! 2 H_0\left(\frac{1}{2}\! -\! \frac{\lceil 2 w \rceil\! -\! 2w}{4}\right), \;\; w > 1.
\end{empheq}
\end{subequations}

Then the sequence $\{x^k\}$ generated by Algorithm~3 converges to a solution $x^* : P(x^*) = 0$.
The values $\|P(x^k)\|$ are monotonically decreasing, at $k$-th step the following estimates for the rate of convergence hold:
\begin{subequations} \label{eq:P_k-and-x_k-alg3-exact} 
\begin{align} 
\label{eq:P_k-alg3-stage1-exact} 
\|P(x^k)\| & \leq \|P(x^0)\| - \frac{\mu^2}{2L}k, \;\;\; k < k_{\max},
\\ 
\label{eq:x_k-alg3-stage1-exact}
\|x^{k}\! -\! x^*\| & \!\! \leq \!\! \frac{\mu}{L}\! \left(\!\!\frac{(4 w_0\! -\! \lceil 2 w_0\rceil \! +\! 3\! -\! k)(\lceil 2 w_0 \rceil \!-\! 2\! -\! k)}{4} \!+\! 2 H_0\Big(\frac{\ol{w}}{2}\Big)\!\!\right)\!, k < k_{\max},
\\ 
\label{eq:P_k-alg3-stage2-exact}
\|P(x^k)\| & \leq \frac{2\mu^2}{L} \Big(\frac{\ol{w}}{2}\Big)^{(2^{(k-k_{\max})})}, \;\;\; k \geq k_{\max},
\\ 
\label{eq:x_k-alg3-stage2-exact}
\|x^{k} \! - \! x^*\| & \leq \frac{2\mu}{L}H_{k - k_{\max}}\Big(\frac{\ol{w}}{2}\Big), \;\;\; k \geq k_{\max}.
\end{align}
\end{subequations} 
where $\ol{w} = \frac{L}{\mu^2}\|P(x^0)\| -\frac{k_{\max}}{2} = \min\big\{\frac{L}{\mu^2}\|P(x^0)\|,$ $1 - \frac{1}{2}\big\lceil \frac{2L}{\mu^2}\|P(x^0)\|\big\rceil + \frac{L}{\mu^2}\|P(x^0)\| \big\} \in [0, 1)$.
\end{theorem}
For proving the theorem we need the simple proposition about real sequences.

\begin{proposition}\label{comparison} 
Consider two non-negative real sequences $w_k \geq 0, v_k \geq 0, \; k \geq 0$, and functions $h_k(v), f(v), k \geq 0$. Let $f(\cdot)$ be monotonically increasing function, being also a majorant function for $h_k(\cdot)$ with respect to $\{w_k\}$. Namely, we require $h_k(w_k) \leq f(w_k)$.
If $w_0 \leq v_0$ and the sequences satisfy
\[
\begin{array}{c}
w_{k+1} \leq h_k(w_k), \\
v_{k+1} = f(v_k), 
\end{array}
\]
then $w_k \leq v_k, \; k \geq 0$.
\end{proposition}
The proposition is trivially proved by induction step $w_{k+1} \leq h_k(w_k) \leq f(w_k) \leq f(v_k) = v_{k+1}$.

The proof of Theorem~\ref{thm:alg3-exact} resembles the proof of Theorem~\ref{thm:alg1-exact}, and it uses majorization idea. Main issue is due to different step-size, now there is no clear separation between damped and pure Newton steps.
\begin{proof}
We compare two discrete processes, both starting with the same value $v_0 = w_0$.
The first sequence is generated by recurrent equality $v_{k+1} = f(v_k)$, where
\[
f(v) = \min_{\alpha} \Big(|1 - \alpha| v + \frac{\alpha^2}{2} v^2\Big)
= \left\{
\begin{array}{ll}
v - \frac{1}{2}, & v > 1, \\
\frac{1}{2} v^2, & v \leq 1,
\end{array}
\right.
\]
is monotonically increasing function on $v \geq 0$.
The second process is $w_k = \frac{L}{\mu^2}\|P(x^k)\|$, with $\{x^k\}$ generated by Algorithm~3.
Assume that all $x^k \in B$ and thus Assumptions~$\mathbf{A'}, \mathbf{B'}$ hold.
Then due to main inequality \eqref{eq:ineq1} and step-size \eqref{eq:alg3-alpha}
\[
w_{k+1} \leq 
h_k(w_k) = \min_{\alpha} \Big(|1-\alpha| w_k + \frac{\alpha^2}{2} \left(\frac{L \|z^k\|}{\mu} \right)^2\Big).
\]
Here we used $\|z^k\|$ in parametric part $a^2_k \geq 0$ within introduced function $h_k(w) = \min_\alpha (|1-\alpha| w + a_k^2 \frac{\alpha^2}{2})$.
From $a_k = \frac{L}{\mu} \|z^k\| \leq \frac{L}{\mu^2} \|P(x^k)\| = w_k$ by Assumption~$\mathbf{B'}$, the functions $|1-\alpha| w + a^2_k \frac{\alpha^2}{2} $ under minimization in $h_k(\cdot)$ are majorized by corresponding functions $|1-\alpha| v + w_k^2 \frac{\alpha^2}{2} \leq |1-\alpha| v + \frac{\alpha^2}{2} v^2$ for $v \geq w_k$.
Minimums of the functions over $\alpha$ (implicitly dependent on $w_k, v$) also satisfy $h_k(v) \leq f(v)$ whenever $v \geq w_k$. It follows that $h_k(w_k) \leq f(w_k)$.

Therefore sequences $\{w_k\}$ and $\{v_k\}$, alongside with functions $h_k(\cdot), f(\cdot)$ satisfy Proposition~\ref{comparison}, given $v_0 = w_0 = \frac{L}{\mu^2}\|P(x^0)\|$.
Now we have upper bound on $w_k$ through $v_k$ for all $k \geq 0$. Next we closely follow the lines and calculations of the proof of
Theorem~\ref{thm:alg1-exact}.

Analysis of $\{v_k\}$ is the same as analysis of the upper bound in Theorem~\ref{thm:alg1-exact}. First, $v_k$ decreases, and the number of steps until $v_k$ reaches $1$ is the same $k_{\max}$, given by \eqref{eq:k_max-definition}. Explicit expressions on $v_k$ are
\[
\begin{array}{ll}
\displaystyle
v_k = v_0 - \frac{1}{2} k, & \; k \leq k_{\max}, \\[3mm]
\displaystyle
v_k = 2 \Big( \frac{v_{k_{\max}}}{2} \Big)^{(2^{k - k_{\max}})}, & \; k > k_{\max}. \\
\end{array}
\]
where $v_{k_{\max}} = \ol{w} = v_0 - \frac{k_{\max}}{2} = \min\{w_0, 1 - \frac{\lceil 2 w_0 \rceil - 2 w_0}{2}\}$ (remind that $v_0 = w_0$ by definition). These expressions result in the bounds \eqref{eq:P_k-alg3-stage1-exact} and \eqref{eq:P_k-alg3-stage2-exact} on $\|P(x^k)\|$, which are the same as in Theorem~\ref{thm:alg1-exact}.

For the late steps with $k \geq k_{\max}$ we have $w_k \leq v_k \leq 1$, and thus $\alpha^*_k = 1$ (just because $\frac{\|P(x^k)\|}{L \|z^k\|^2} \geq \frac{\mu^2}{L \|P(x^k)\|} = \frac{1}{w_k} \geq 1$).
Then points $x^{k}, k \geq k_{\max}$ form a Cauchy sequence like \eqref{eq:x_a-x_b-alg1-stage2}, and obey
similar to \eqref{eq:x_k-x*-alg1-stage2} bound:
\[
\|x^{k_{\max} + \ell} - x^*\| 
{\color{red} \leq } \sum_{i = \ell}^{\infty} \|z^{\color{red}k_{\max} + i}\| 
\leq \frac{\mu}{L} \sum_{i = \ell}^{\infty} w_{\color{red}k_{\max} + i}
\leq \frac{\mu}{L} \sum_{i = \ell}^{\infty} v_{\color{red}k_{\max} + i} 
= \frac{2 \mu}{L} H_\ell \left(\frac{\ol{w}}{2}\right), \; \ell \geq 0. 
\]
This is \eqref{eq:x_k-alg3-stage2-exact} bound, by the way the same as \eqref{eq:x_k-alg1-stage2-exact} of Theorem~\ref{thm:alg1-exact}.

What is the main difference from Theorem~\ref{thm:alg1-exact} proof, is the distance counting until $k_{\max}$-th step. In this case we assume $k_{\max} = \lceil 2 w_0 \rceil - 2 > 0$. Due to the Algorithm's step-size choice now we cannot be sure, whether $\alpha^*_k$ is always less than $1$ or not. For $i < k_{\max}$ we have $\|x^{i+1} - x^i\| = \alpha^*_i \|z^i\| \leq \|z^i\| \leq \frac{1}{\mu} \|P(x^i)\| = \frac{\mu}{L} w_i \leq \frac{\mu}{L} v_i = \frac{\mu}{L}(v_0 - \frac{1}{2}i)$, and arrive to \eqref{eq:x_k-alg3-stage1-exact}:
\begin{equation} \nonumber
\begin{array}{rl}
\|x^{k} - x^*\| \leq & \displaystyle
\|x^{k_{\max}} - x^*\| + \sum_{i = k}^{k_{\max}-1} \|x^{i+1} - x^i\| \leq \\
\leq & \displaystyle 
2\frac{\mu}{L} H_0\Big(\frac{\ol{w}}{2}\Big) + 
\frac{\mu}{L}\sum_{i=k}^{k_{\max} - 1} \Big(v_0 - \frac{1}{2}i\Big) = \\
= & \displaystyle
\frac{\mu}{L} \left(\frac{(4 v_0 - k_{\max} + 1 - k)(k_{\max} - k)}{4} + 2 H_0\Big(\frac{\ol{w}}{2}\Big) \right) = \\
= & \displaystyle
\frac{\mu}{L} \left(\frac{(4 w_0 - \lceil 2 w_0\rceil + 3 - k)(\lceil 2 w_0 \rceil - 2 - k)}{4} + 2 H_0\Big(\frac{\ol{w}}{2}\Big) \right).
\end{array}
\end{equation}
In \eqref{eq:x_k-alg3-stage1-exact} we used explicit {\color{red}formula} for $\ol{w} = 1 + w_0 - \frac{\lceil 2 w_0 \rceil}{2}$ in case $k_{\max} > 0$.

The last part of the proof is checking assumption $x^k \in B$, i.e. $\|x^0 - x^k\| \leq \rho$. From the {\color{red}derivation of bound \eqref{eq:x_k-alg3-stage1-exact} above} we have
\[
\begin{array}{rl}
\|x^0 - x^k\| \leq & \displaystyle 
{\color{red}\sum_{i=0}^{k-1} \|x^{i+1} - x^i\| 
\leq \sum_{i=0}^{k_{\max}-1} \|x^{i+1} - x^i\| + \sum_{i=k_{\max}}^{\infty} \|x^{i+1} - x^i\|} \leq \\
\leq & \displaystyle
\frac{\mu}{L} \left(\frac{(4 w_0 - \lceil 2 w_0\rceil + 3)(\lceil 2 w_0 \rceil - 2)}{4} + 2 H_0\Big(\frac{\ol{w}}{2}\Big) \right), \; k \geq 0,
\end{array}
\]
in case $k_{\max} > 0$, i.e. $w_0 \geq 1$, and $\|x^{\color{red}0} - x^{\color{red}k}\| \leq 2 \frac{\mu}{L} H_0(\frac{w_0}{2})$ otherwise {\color{red}(from derivation of bound \eqref{eq:x_k-alg3-stage2-exact})}. Thus sufficient condition for $x^k \in B$ is
\[
\rho \geq \frac{\mu}{L} F_2\Big(\frac{L}{\mu^2}\|P(x^0)\|\Big),
\]
which is equivalent to \eqref{eq:cond-P_0-alg3}.
\end{proof}
Like for Algorithm~1, we can state few corollaries: on global convergence and some simple bounds.

\begin{corollary}
If $\rho = \infty$ (that is conditions $\mathbf{A'}, \mathbf{B'}$ hold on the entire space $\R^n$) 
then Algorithm~3 converges to a solution of \eqref{eq:main-nonlinear-P} for any $x^0 \in \R^n$.
\end{corollary}

There is a simpler tight relaxed condition for Theorem~\ref{thm:alg3-exact}. 
\begin{corollary} \label{cor:alg3-simple-bound}
Condition {\color{red}\eqref{eq:cond-P_0-alg3}} can be replaced with 
\begin{equation} \nonumber
\begin{array}{rl}
\|P(x^0)\| \leq & \displaystyle \frac{\mu^2}{L}F_2^{\inv, \lw}\Big( \frac{L}{\mu}\rho \Big) = \\
= & \displaystyle \frac{\mu^2}{L} \times \left\{
\begin{array}{cl}
\displaystyle
\frac{1}{4} - 2 c_3 + \sqrt{\Big(2 c_3 - \frac{1}{4}\Big)^2 + \frac{2L}{\mu}\rho}, 
& \displaystyle \;\; 0 \leq \rho \leq c_3\frac{\mu}{L}, \\[4mm]
\displaystyle
-\frac{1}{4} + \sqrt{\frac{L}{\mu}\rho - c_3 + \frac{9}{16}},
& \displaystyle \;\; \rho > c_3\frac{\mu}{L},
\end{array}
\right.
\end{array}
\end{equation}
where constant $c_3 = 2 c_1 + c_2 - 1 \approx 0.66885$. 
\end{corollary}
The idea of the proof is the same as of Corollary~\ref{cor:alg1-simple-bound}: we derive upper bound $F_2^{\up}(w)$ for function $F_2(w)$ \eqref{eq:cond-p-alg3}. Then the sufficient condition for points $x^k$ generated by Algorithm~3 being inside $B$ is $\rho \geq \frac{\mu}{L} F_2^{\up}(\frac{L}{\mu^2} \|P(x^0)\|)$. Its inverse function $F_2^{\inv, \lw} : F_2^{\inv, \lw}(F_2^{\up}(w)) \equiv w, w \geq 0$ is a lower bound for $F_2^{\inv}(\cdot)$ then.

\begin{proof}
Notice that both components of $F_2(w)$ are the same in $w \in [\frac{1}{2}, 1]$, so
\begin{subequations} \nonumber
\begin{empheq}[left={\displaystyle \!\!\!F_2(w) \!=\!\! \empheqlbrace}]{align}
\nonumber
\displaystyle
& 2 H_0\Big(\frac{w}{2}\Big), \hspace{58mm} 0 \leq w \leq \frac{1}{2}, \\
\nonumber
& \frac{(\lceil 2w \rceil \! - \! 2) (4 w \! - \!\lceil 2w \rceil \! +\! 3)}{4} 
\! +\! 2 H_0\left(\frac{1}{2}\! -\! \frac{\lceil 2 w \rceil\! -\! 2w}{4}\right), \;\; w > \frac{1}{2}.
\end{empheq}
\end{subequations}
Let's begin with the case $w > \frac{1}{2}$. We introduce the auxiliary function $r(w) = \frac{2w - \lceil 2w \rceil}{4} \in (0, \frac{1}{4}]$, related with the fractional part. This function is periodic on $w$, and
\[
\begin{array}{rl}
F_2(w) = & \displaystyle \Big(w + \frac{1}{4}\Big)^2 + 2 H_0\Big(\frac{1}{2} - r(w) \Big) - \Big(\frac{5}{4} - 2 r(w) \Big)^2 \leq \\[3mm]
\leq & \displaystyle \Big(w + \frac{1}{4}\Big)^2 + 2 c_1 + c_2 - \frac{25}{16}, \;\; w > \frac{1}{2}.
\end{array}
\]
Here we used definition \eqref{eq:c_2-define} of constant $c_2$, introduced in Section~\ref{sec:preliminaries}.

By the definition of $H(\cdot)$ one can select terms up to quadratic in $H_0(\delta) = \delta + \delta^2 + H_2(\delta)$, thus $H_2(\frac{1}{4}) = H_0(\frac{1}{4}) - \frac{5}{16} = H_0(\frac{1}{2}) - \frac{1}{2} -\frac{5}{16} = c_1 - \frac{13}{16}$. 
From convexity we have the upper linear bound for $H_2(\delta) \leq (4 c_1 - \frac{13}{4}) \delta, \;\; \delta \in [0, \frac{1}{4}]$, and consequently for $2 H_0(\frac{w}{2})$: 
\[
2 H_0\Big(\frac{w}{2} \Big) =
2\frac{w}{2} + 2 \frac{w^2}{4} + H_2(\frac{2}{2}) \leq
\frac{w^2}{2} + \Big(4 c_1 - \frac{9}{4}\Big) w \leq \frac{w^2}{2} + \Big(4 c_1 + 2 c_2 - \frac{9}{4}\Big) w, \;\; 0 \leq w \leq \frac{1}{2}.
\]
In the last inequality we manually added a small positive linear term $2 c_2 w$ for continuity of the resulting upper bound. 
Combining two parts, we arrive to the increasing continuous upper bound for $F_2(\cdot)$:
\[
F_2(w) \leq F_2^{\up}(w) = \left\{
\begin{array}{cl}
\displaystyle
\frac{w^2}{2} + \Big(4 c_1 + 2 c_2 - \frac{9}{4}\Big) w, 
& \displaystyle
0 \leq w \leq \frac{1}{2}, \\
\displaystyle 
\Big(w + \frac{1}{4}\Big)^2 + 2 c_1 + c_2 - \frac{25}{16}, 
& \displaystyle 
w > \frac{1}{2}.
\end{array}
\right.
\]
Using definition of $c_3 = F_2^{\up}(\frac{1}{2}) = 2 c_1 + c_2 - 1$, inverse of this function is the lower bound for $F_2^{\inv}(\cdot)$
\[
F_2^{\inv}(p) \geq F_2^{\inv,\lw}(p) = 
\left\{
\begin{array}{cl}
\displaystyle
\frac{1}{4} - 2 c_3 + \sqrt{\Big(2 c_3 - \frac{1}{4}\Big)^2 + 2 p}, 
& \displaystyle \;\; 0 \leq p \leq c_3, \\[4mm]
\displaystyle
-\frac{1}{4} + \sqrt{p - c_3 + \frac{9}{16}},
& \displaystyle \;\; p > c_3.
\end{array}
\right.
\]
\end{proof}

The presented bounds are sharp and quite exact. Visually paired graphics of $F_2(w)$ and $F_2^{\up}(w)$, $F_2^{\inv}(p)$ and $F_2^{\inv, \lw}(p)$ looks the same, and its residuals are plotted on Figure~\ref{fig:alg3-w-on-p-residual}.

\begin{figure} 
\caption{Algorithm 3: residual of upper bound $F_2^{\up}(w) - F_2(w)$ (left) and residual of lower bound $F_2^{\inv}(p) - F_2^{\inv, \lw}(p)$ (right).}
\label{fig:alg3-w-on-p-bounds}
\includegraphics[width=0.5\linewidth]{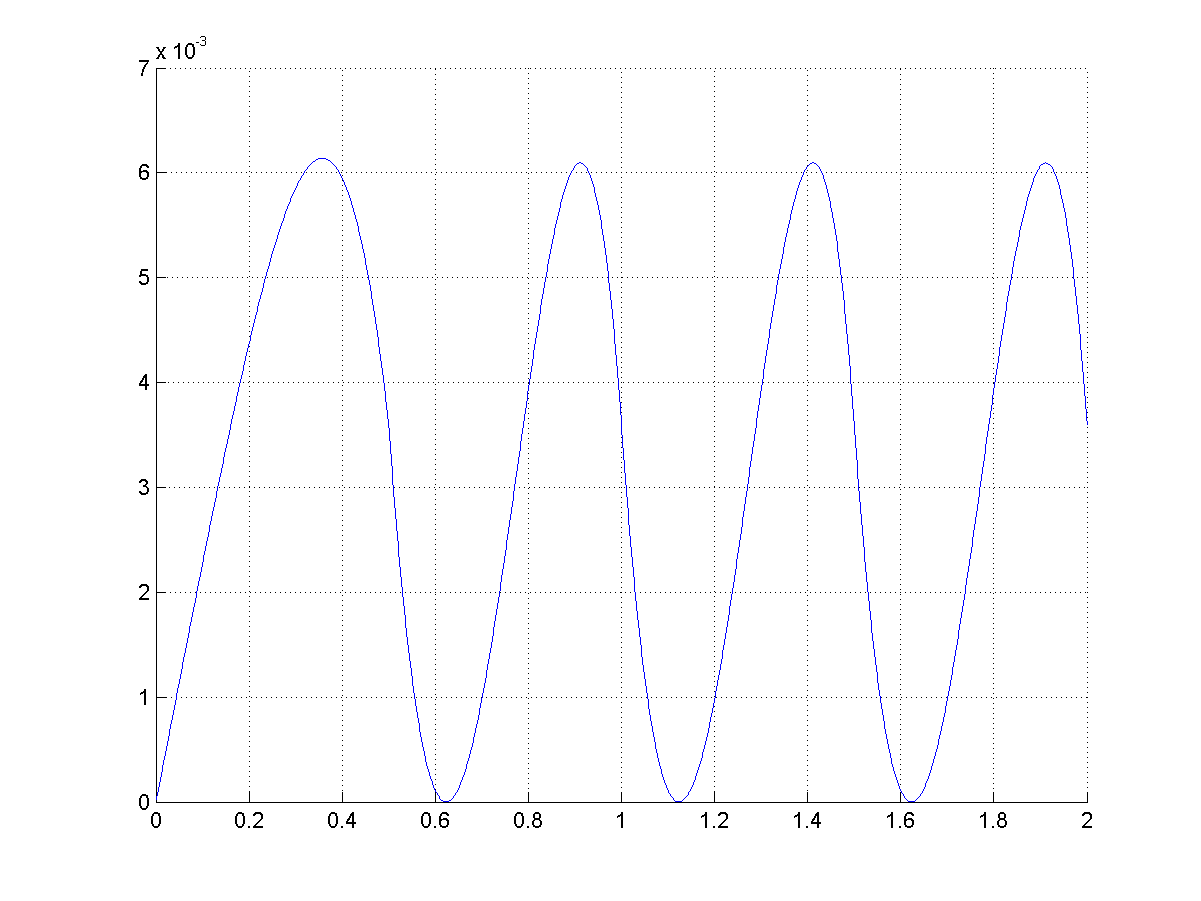}
\includegraphics[width=0.5\linewidth]{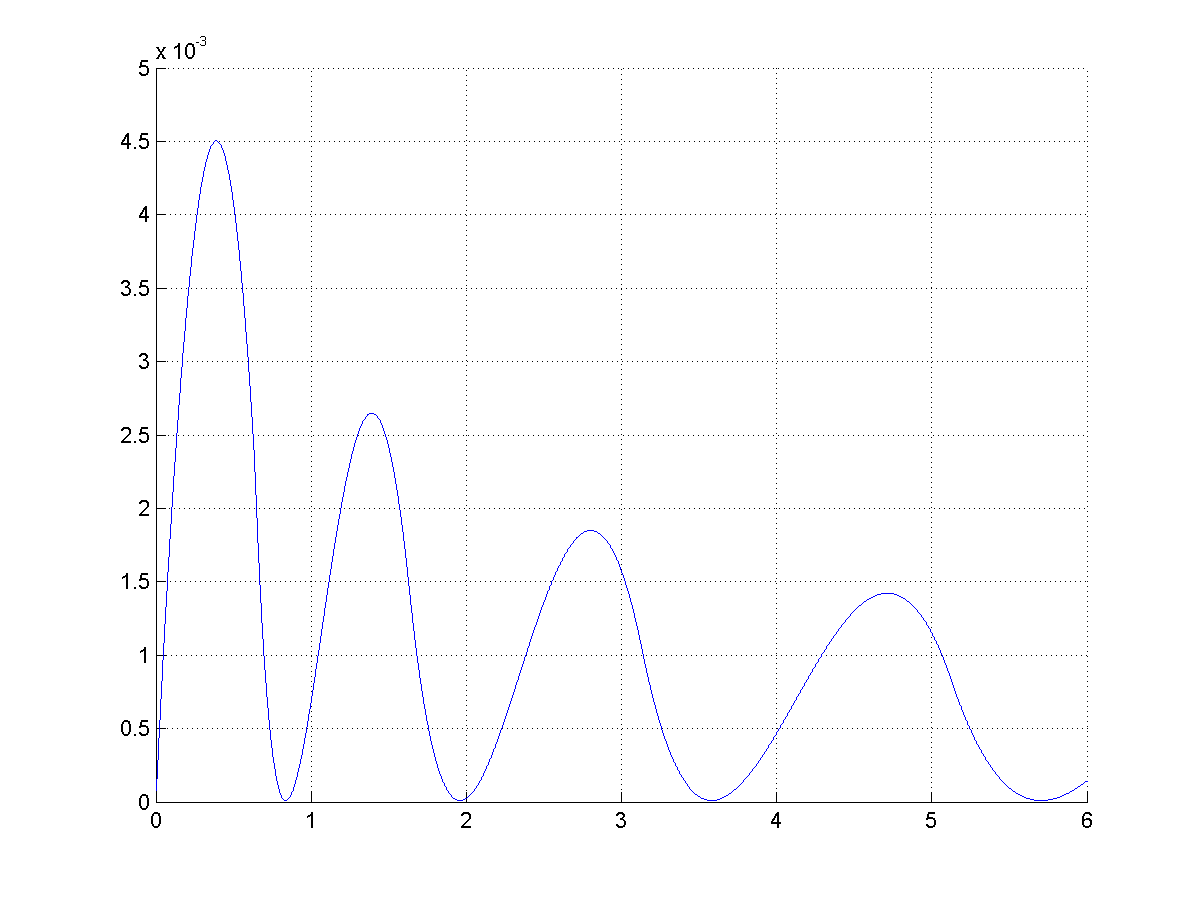}
\label{fig:alg3-w-on-p-residual}
\end{figure}

Surprisingly enough, in practice the Algorithm~3 (and its adaptive modification) sometimes converges faster than Algorithm~1, possibly because direction-wise (along $z^k$) Lipschitz constant is less or equal than uniform Lipschitz constant of Assumption~$\mathbf{A'}$, and the convergence rate can be better.

The idea of adaptive algorithm with estimates $L_k$ works as well for Algorithm~3; including its modifications with increasing $L_k$.

\subsection{Pure Newton method}
For comparison let us specify convergence conditions of pure Newton method ($\alpha_k = 1$).
\begin{theorem}
Let conditions $\mathbf{A'}, \mathbf{B'}$ hold. If $\delta = \frac{L}{2\mu^2}\|P(x^0)\| < 1$ and $\frac{2\mu}{L}H_0(\delta) \leq \rho$, then pure Newton method converges to a solution $x^*$ of \eqref{eq:main-nonlinear-P}, and
\[
\|P(x^k)\| \leq \frac{2\mu^2}{L} \delta^{(2^k)}, \;\;
\|x^k - x^*\| \leq \frac{2\mu}{L} H_k(\delta).
\]
\end{theorem}
It coincides with Corollary~1 of \cite{Polyak-1964}, proven in the Banach space setup (a misprint in \cite{Polyak-1964} is corrected here). For $m = n$ case the result is a minor extension of Mysovskikh's theorem \cite{Kant82}.

\section{Special Cases}
\label{sec:special-cases}

In the section we outline few important cases in more detail, namely solving equations with special structure, solving scalar equations or inequalities, solvability of quadratic equations.

\subsection{Structured problems}

The problem is to solve equation $g(x)=y$ where $g(x)_i=\varphi(c_i^T x), \; c_i\in R^n, i=1,\dots m$. Here $\varphi (t)$ is a twice differentiable scalar function, 
\[ |\varphi'(t)| \geq \mu_\varphi> 0, \;\; |\varphi''(t)|\leq L_\varphi, \;\; \forall t.\] 

It is not hard to see that Assumptions $\mathbf{A}, \mathbf{B}$ hold on the entire space $\R^n$ and Algorithm~1 converges, with Theorem~\ref{thm:alg1-exact} and Corollary~\ref{cor:alg1-simple-bound} providing rate of convergence. The rate of convergence depends on estimates for $\mu, L$, which can be written as functions of $\mu_\varphi, L_\varphi$ and minimal and maximal singular values $\sigma_{\min}, \sigma_{\max}$ of matrix $C$ with columns $c_i$ (we suppose that $C$ has full rank, thus $\sigma_{\min}>0$). Indeed after simple calculations {\color{red}(see expression for $g'(x)$ with $C$ below)} we get
\begin{equation}\label{struct}
L\leq {\color{red}\sigma_{\max}^2} L_\varphi, \quad \mu\geq \sigma_{\min} \mu_\varphi.
\end{equation}

However the special structure of the problem allows to get much sharper results.
{\color{red}
Let's use notation $P(x) = g(x) - y$. Indeed $P'(x) = g'(x) = D(x)C^T,\; D(x)= \diag (\varphi'(c_i^T x))$ and repeating the proof of Theorem 3.2 we get the equality
\[
P(x^{k+1}) = |1 - \alpha| P(x^k) - \alpha \int_{0}^{1}(D_t -D)C^T z^kdt, \; \alpha \geq 0,
\]
where $D_t = D(x^k - \alpha tz^k), D = D(x^k)$.
Thus (recall $u_k = \|P(x^k)\|$)
\[
u_{k+1} 
\leq |1 - \alpha| u_k + \alpha  ||C^T z^k|| \int_{0}^{1}||D_t - D||d t 
\leq |1 - \alpha| u_k + \frac{L_\varphi \alpha^2 ||C^T z^k||^2}{2}
\]
Identity between spectral norm of a diagonal matrix and Euclidean norm 
of vector on the diagonal is used in the last line, 
followed by element-wise Lipschitz property of $\varphi'(\cdot)$:
$\|D_t - D\| = \|[\varphi'(c_i^T (x^{k} - \alpha tz^k)) - \varphi'(c_i^T x^k)]\|$
$\leq \|[ \big|\varphi'(c_i^T (x^{k} - \alpha tz^k)) - \varphi'(c_i^T x^k)\big| ]\|$
$\leq \|[ L_\varphi \alpha t \big| c_i^T z^{k}\big| ]\| = L_\varphi \alpha t \|[ c_i^T z^{k} ]\| = L_\varphi \alpha t \|C^T z^k\|$ for $t \geq 0$.
Thus $\int_{0}^{1}||D_t - D||d t \leq \frac{1}{2}L_\varphi \alpha \|C^T z^k\|$.

But $P'(x^k) z^k=P(x^k)$, thus $D C^T z^k = P(x^k),\; C^T z^k = D^{-1} P(x^k)$ and hence 
$||C^T z^k|| \le \frac{u_k}{\mu_\varphi}$. 
We arrive to the inequality, very similar to \eqref{eq:ineq2}, 
but with different constant $\gamma$
\[ 
u_{k+1} \leq |1 - \alpha| u_k + \gamma\frac{\alpha^2 u_k^2}{2}, \;\; 
\gamma = \frac{L_\varphi}{\mu_\varphi^2}.
\]
Hence $u_{k+1} \leq u_k - \frac{1}{2\gamma}$ at Stage 1, thus this inequality does not depend on $C$! 
As the result we get estimates for the rate of convergence 
which are the same for ill-conditioned and well-conditioned matrices $C$. 
Of course this estimate is much better than standard one 
with $\gamma= \big(\frac{\sigma_{\max}}{\sigma_{\min}} \big)^{\!2} \frac{L_\varphi}{ \mu_\varphi^2}$ 
which follows from \eqref{struct}.
}

This example is just an illustrating one (explicit solution of the problem can be found easily), but it emphasizes the role of special structure in equations to solve. Numerical experiments with such problems are provided below, in Section 6.3.

\subsection{One-dimensional case}

Suppose we solve one equation with $n$ variables: \[f(x)=0, \;\; f: \R^n \rightarrow \R.\]

Here $0$ is \emph{not a minimal value} of $f$, thus it is not a minimization problem! Nevertheless our algorithms will remind some minimization methods. This case has some specific features compared with arbitrary $m$. For instance calculation of $z^k$ may be done
explicitly. Norm in image space is absolute value $|\cdot|$, and $\ell_p$ norms in pre-image space $\R^n, \; p \in \{1, 2,
\infty\}$ can be considered.
Then 
\[
\begin{array}{ll}
\displaystyle
z^k = \frac{f(x^k) \sign(\nabla f(x^k)_i)}{\|\nabla f(x^k)\|_\infty} e^i, \; i \in \Arg \max_i |\nabla f(x^k)_i|, & \text{ in case of } \ell_1\text{-norm},\\
\displaystyle
z^k = \frac{f(x^k)}{\|\nabla f(x^k)\|^2_2} \nabla f(x^k), & \text{ in case of Euclidean norm},\\
\displaystyle
z^k = \frac{f(x^k)}{\|\nabla f(x^k)\|_1} \sign(\nabla f(x^k)), & \text{ in case of } \ell_\infty\text{-norm},
\end{array}
\]
where $e^j = (0, \ldots, 0, 1, 0, \ldots, 0)^T$ is $j$-th orth vector, and $\sign(\cdot)$ function is coordinate-wise sign function, $\sign : \R^n \rightarrow \{-1, 1\}^n$.

Constant $\mu$ (and $\mu_0$) are also calculated explicitly via conjugate (dual) vector norm $\mu = \min_{x \in B} \|\nabla f(x)\|_*$, $\mu_0 = \|\nabla f(x^0)\|_*$. For any norms $\|z^k\|=|f(x^k)| / \|\nabla f(x^k)\|_*$, and in Algorithm~3 damped Newton step is performed iff $\|\nabla f(x^k)\|^2_* < L |f(x^k)|$, otherwise pure Newton step is made.

If we choose $\ell_1$ norm, the method becomes coordinate-wise one. Thus, if we start with $x^0=0$ and perform few steps (e.g. we are in the domain of attraction of pure Newton algorithm) we arrive to a \emph{sparse} solution of the equation.

In Euclidean case a Stage 1 step (damped Newton) of Algorithm~3 is 
\[ x^{k+1} = x^k - \frac{1}{L} \sign(f(x^k)) \nabla f(x^k), \] 
which is exactly gradient minimization step for function $|f(x^k)|$. Stage~2 (pure Newton) step is 
\[ x^{k+1} = x^k - \frac{f(x^k) }{\|\nabla f(x^k)\|^2_2} \nabla f(x^k).\]
This reminds well-known subgradient method for minimization of convex functions. However in our case we do not assume any convexity properties, and the direction may be either gradient or anti-gradient in contrast with minimization methods!

\subsection{Quadratic equations} \label{ssec:quadratic}

Proceed to a specific nonlinear equation, namely the quadratic one. Then the function $g(x)$ may be written componentwise as \eqref{eq:quadratic-componentwise},
with gradients \[ \nabla g_i(x) = A_i x + b_i \in \R^n, \;\; i = 1,\ldots, m. \] Obviously $g(0)=0$, the question is solvability of $g(x)=y$. There are some results on construction of the entire set of feasible points $Y=\{y: g(x)=y \}=g(\R^n)$, including its convexity, see e.g. \cite{Polyak1998}. We focus on local solvability, trying to derive the largest ball inscribed in $Y$.

The derivative matrix $g'(x)$ is formed row-wise as
\begin{equation} \nonumber \label{eq:quadratic-grad-g}
g'(x) = \left[
\begin{array}{c}
\nabla g_1(x)^T \\
\vdots \\
\nabla g_m(x)^T \\
\end{array}
\right]
= \left[
\begin{array}{c}
x^T A_1 + b_1^T \\
\vdots \\
x^T A_m + b_m^T \\
\end{array}
\right] \in \R^{m \times n}.
\end{equation}

One has $g'(0)=H, \; H$ being $m\times n$ matrix with rows $b_i$. We suppose $H$ has rank $m$ (recall that $m\leq n$), then its smallest singular value $\sigma_{\min}(H) > 0$ serves as $\mu_0$. 

The derivative $g'(x)$ is linear on $x$, thus it has uniform Lipschitz constant $L$ on $\R^n$, and assumption $\mathbf{A}$ holds everywhere. There are several estimates for the Lipschitz constants, for example (for $\ell_2$ norm)

\[ L\le L_1 = \sqrt{\lambda_{\max}\left(\sum_{i = 1}^m A_i^T A_i\right)} \] from \cite{Polyak-2001}, where $\lambda_{\max}$ is the maximal eigenvalue of a matrix.
Other estimates can be obtained via elaborate convex semidefinite optimization problem (SDP), cf. \cite{Xia-2014} for details. 

Quadratic equations play significant role in power system analysis, because power flow equations are
quadratic, see \cite{Bialek}. It is of interest to compare our estimates with some known results on solvability of power flow equations \cite{Turitsyn}.

\subsection{Solving systems of inequalities}

Below we address some tricks to convert systems of inequalities into systems of equations.

First, if one seeks a solution of a system of inequalities \[ g_i(x) \leq 0, \; i = 1,\ldots, m, \;\; x \in \R^\ell, \]  then by introducing slack variables the problem is reduced to solution of the under-determined system of equations
 \[ g_i(x) + x_{\ell + i}^2 = 0, \; i = 1,\ldots, m, \;\; x \in \R^{n}, n=\ell+m. \]

Similarly finding a feasible point for linear inequalities $x \geq 0, A x = b, \;\; x \in \R^n, \; b \in \R^m$ can be transformed to the under-determined system 
\[ \sum_{j = 1}^n A_{ij} z_i^2 = b_i, \; i = 1,\ldots, m, \;\; z \in \R^{n }. \]

The efficiency of such reductions is unclear a priori and should be checked by intensive numerical study.

\section{Numerical tests}
\label{sec:numerical}

\newcommand{\ExpOneConstNFunctions}{100}
\newcommand{\ExpOneConstNInitPoints}{1000}

\newcommand{\ExpOneConstInitBeta}{100}
\newcommand{\ExpOneConstLambda}{0.95}
\newcommand{\ExpOneConstC}{0.8}
\newcommand{\ExpOneConstFTol}{10^{-8}}
\newcommand{\ExpOneConstGammaTol}{10^{-13}}
\newcommand{\ExpOneConstNMaxSteps}{10000}

We have performed several experiments to check effectiveness of the proposed approach for solving equations \eqref{eq:main-nonlinear-P} and to compare it with known ones.

The first two experiments relate to the classical case $n=m$, i.e. the number of variables equals the number of equations. Algorithm~2 with adaptive parameter estimation is compared
with well-known ``backstepping'' Armijo-like techniques for the damped Newton method.
Namely, the competitor is the step-size proposed in \cite{Burdakov}.
\begin{equation} \label{eq:armijo-stepsize}
\gamma_k = q^{j} : \|P(x^k + q^j z^k)\| \leq (1 - c q^j )\|P(x^k)\|
\end{equation}
with some shrinkage parameter $q \in (0, 1)$ and relaxation parameter $c \in (0, 1)$.
This method in some sense is similar to our Algorithm~2, but there are differences in step-size choice.

Two other examples relate to under-determined case, i.e. $n>m$. Example 3 is the illustration how to employ structure of the data as explained in Subsection 5.1. We show that such approach strongly accelerates convergence. Final Example 4 is an optimal control problem. Exploiting $L_1$ norms we construct \textit{sparse} controls for the minimal-fuel problem.

\subsection{Example 1}

We studied the Fletcher-Powell system of equations \cite{Fletcher-Powell}:
\begin{equation} \label{eq:fletcher-powell}
\sum_{j=1}^n A_{i j} \sin x_j + B_{i j} \cos x_j = E_i, \; i = 1, 2,...,n
\end{equation}
for various dimensions $n$. The data are generated as proposed in original paper \cite{Fletcher-Powell}:
matrices $A, B \in \R^{n\times n}$ are random, then for some $x^* \in \R^n$ right-hand sides $E \in \R^n$ are calculated. The arising system of equations may have many solutions, however we have the guarantee that it is solvable.
Then a set of 1000 initial points $x^0$ were randomly sampled (multistart policy). 
From each of the initial points, Newton algorithms were run with a) Armijo-kind step-size \eqref{eq:armijo-stepsize}, and b) $\beta$-adaptive algorithm (Algorithm 2). 
Parameters of the algorithms were chosen as $\beta_0 = \ExpOneConstInitBeta, q = \ExpOneConstLambda, c = \ExpOneConstC$.

Each run has a ``success'' or ``fail'' result. ``Success'' means that accuracy $||P(x^k)||<10^{-8}$ is achieved, while ``failure'' is marked when either step-size threshold is $\gamma_k<10^{-13}$ attained or maximal number of iterations ($N=10000$) is performed.
For each dimension $n$ data of algorithms' outcomes were aggregated as following.
For a random sample of equations (there were $100$ for each $n$) ``success ratio'' $r$ was calculated among $\ExpOneConstNInitPoints$ initial points as ratio of success runs and total number of runs: $N_{success}/\ExpOneConstNInitPoints$, both for Armijo-like approach ($r_{Armijo}$), and for our Algorithm~2 ($r_{Alg.~2}$).
To emphasize comparison, we used ratio of ratios $r_{Alg.~2}/r_{Armijo}$ as indicator. 
Then $r_{Alg.~2}/r_{Armijo}$ values were imaged as the box-and-whisker plot for all dimensions (Figure~\ref{fig:exp-1-ratios-whisker}). The middle line in a (quartile) box is the median, whiskers' lengths are set to $0.05$ and $0.95$ percentiles, outlier data is dot-plotted -- as soon there are $100$ points (for each of the samples), there are exactly 5 upper and 5 lower outliers).

\begin{figure}[!h]
\caption{Box-and-whisker plot for the ratios of success ratios $r_{Alg.~2}/r_{Armijo}$ for all dimensions (Ex. 1).}
\includegraphics[width=14cm]{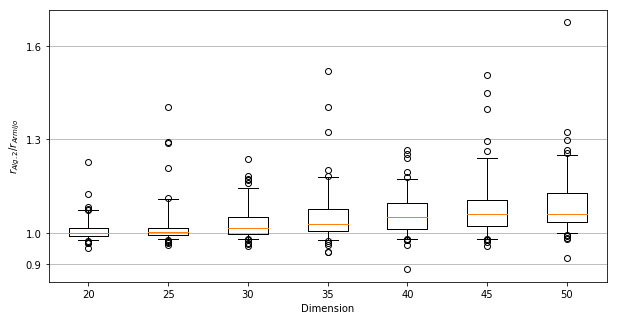}
\label{fig:exp-1-ratios-whisker}
\end{figure}
We see that with high probability the ratio is larger than 1 for large dimensions.
As a conclusion, our method finds a solution more often than Armijo-like approach. 

The similar analysis was done for the function calls. For each of the samples the numbers of function calls 
(these can be many in one Newton step) are averaged over initial conditions for all runs, resulting in values $N_{Alg.~2}, N_{Armijo}$. The ratios of the averaged function calls for Armijo-like step-size \eqref{eq:armijo-stepsize} and Algorithm~2 ($N_{Armijo}/N_{Alg.~2}$) were aggregated on the box-and-whisker plot on Figure~\ref{fig:exp-1-fcalls-whisker}. Typically Algorithm~2 admits much less functions evaluations compared with Armijo-like algorithm. 

\begin{figure}[!h]
\caption{Box-and-whisker plot for the ratios of function calls $N_{Armijo}/N_{Alg.~2}$ for all dimensions (Ex. 1).}
\includegraphics[width=14cm]{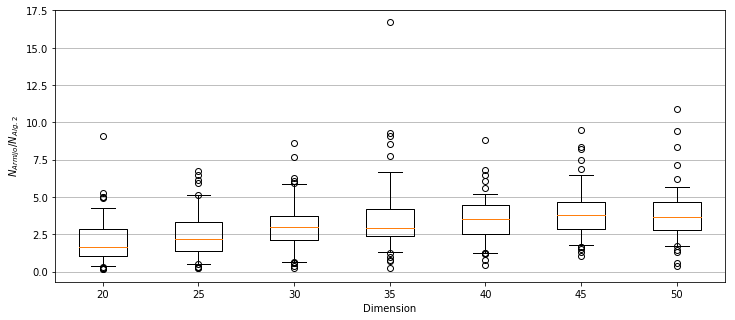}
\label{fig:exp-1-fcalls-whisker}
\end{figure}

\subsection{Example 2}

\newcommand{\ExpTwoConstNFunctions}{100}
\newcommand{\ExpTwoConstNInitPoints}{1000}

\newcommand{\ExpTwoConstInitBeta}{100}
\newcommand{\ExpTwoConstLambda}{0.85}
\newcommand{\ExpTwoConstC}{0.8}
\newcommand{\ExpTwoConstFTol}{10^{-8}}
\newcommand{\ExpTwoConstGammaTol}{10^{-13}}
\newcommand{\ExpTwoConstNMaxSteps}{1000}
The original problem is equality-constrained optimization:
$$\min_{x : h(x) = 0} f(x)$$ 
with scalar differentiable functions $f, h$, and $x \in R^n$.
By use of Lagrange multiplier rule it is reduced to the solution of equations 
$$
P(X) = 
\begin{bmatrix}
\nabla f(x) + \nu \nabla h(x) \\
h(x)
\end{bmatrix}
 = 0
$$
with new variable $X = [x^T, \nu]^T \in R^{n+1}$. Note that we are interested at any solution of these equations, that is we do not distinguish minimum points and stationary points.
The derivative is a block matrix
$$
P'(X) = 
\begin{bmatrix}
\nabla^2 f(x) + \nu \nabla^2 h(x) & \nabla h(x) \\
(\nabla h(x))^T & 0
\end{bmatrix}
$$

We address the simplest case: minimization of a quadratic function (with symmetric~$A$) on Euclidean unit sphere:
$$\min_{||x||^2 = 1} \frac{1}{2} (A x, x) + (b, x)$$
Let the constraint be defined by $h(x) = \frac{1}{2}(x, x) - \frac{1}{2}$, then
\begin{equation} \label{eq:lagrangian-stationary}
P(X) = 
\begin{bmatrix}
A x + b + \nu x \\
\frac{1}{2} x^T x - \frac{1}{2}
\end{bmatrix}, \;\;
P'(X) = 
\begin{bmatrix}
A + \nu I_n & x \\
x^T & 0
\end{bmatrix}.
\end{equation}

The experiment were run for different dimensions ($n=20,25,30,35,40,45,50$).
For each dimension $\ExpTwoConstNFunctions$ problems were randomly generated,
and $\ExpTwoConstNInitPoints$ initial points were randomly chosen for each problem.
Then Algorithm~2 and Armijo-like step-size algorithm \eqref{eq:armijo-stepsize} were run as in the first example. The parameters and stopping criteria were the same, except for the parameter $q = \ExpTwoConstLambda$, and number of Newton steps were bounded by $\ExpTwoConstNMaxSteps$.
Matrix $A$ of the quadratic objective function is a positive {\color{red}semi}definite matrix, formed as $A = \frac{1}{2} M M^T.$
The auxiliary matrix $M \in \R^{n \times n+4}$ has coefficients, uniformly distributed on $[-0.9, 2.1]$.
Coefficients of the linear term $b_i$ are picked up from the scaled Gaussian variables: $b_i \sim 0.1 \mathcal{N}(0, 1)$.
The initial conditions were chosen for the extended variable $X = [x^T, \nu]^T$ as following:
\begin{itemize}
\item 
the first $n$ components ($x^0$, corresponding to the original variable $x$) are sampled from the uniform distribution on Euclidean sphere with radius $1$, 
\item 
the last, $n+1$ component (Lagrange multiplier $\nu$) is chosen as the ``best approximation'', i.e. as the min{\color{red}i}mizer of the residual $\|P(X)\|^2 = \|A x^0 + b + \nu x^0\|^2$. The explicit solution depends on the first $n$ components ($x^0$),
$$
\nu_0 = - (x^0)^T A x^0 - (x^0)^T b.
$$
\end{itemize} 

Same as in Example 1, success ratios $r_{Armijo}, r_{Alg.~2}$ and number of function calls $N_{Armijo}, N_{Alg.~2}$ were calculated (averaged over initial points). On Figure~\ref{fig:exp-2-ratios-whisker} the ratios of $r_{Alg.~2} / r_{Armijo}$, and on Figure~\ref{fig:exp-2-fcalls-whisker} the ratios $N_{Armijo}/N_{Alg.~2}$ were gathered in box-and-whisker plot.

\begin{figure}[!h]
\caption{Box-and-whisker plot for the ratios of success ratios $r_{Alg.~2}/r_{Armijo}$ for all dimensions (Ex. 2).}
\includegraphics[width=14cm]{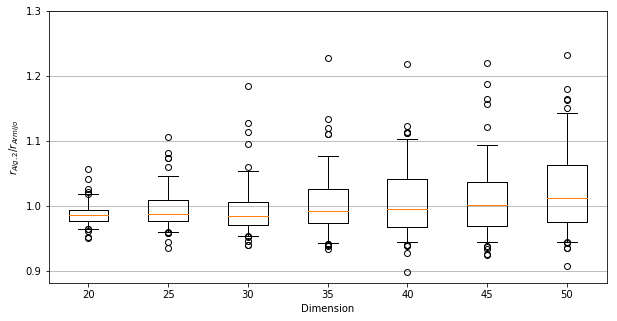}
\label{fig:exp-2-ratios-whisker}
\end{figure}

\begin{figure}[!h]
\caption{Box-and-whisker plot for the ratios of function calls $N_{Armijo}/N_{Alg.~2}$ for all dimensions (Ex. 2).}
\includegraphics[width=14cm]{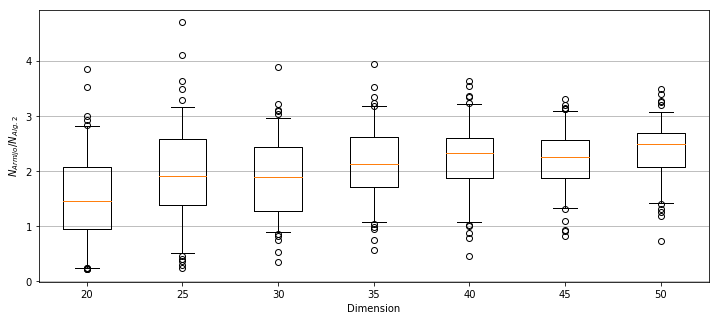}
\label{fig:exp-2-fcalls-whisker}
\end{figure}

Here it appears, that the success ration of Algorithm~2 prevails over Armijo-like approach only at dimension $50$ (still the ratio is about $1$, meaning that both algorithms behave quite similar); however, the average of function evaluations of our algorithm still lower for all cases.

\subsection{Example 3}
The problem is described in Section 5.1; it is to solve $g(x)=y$ where $g(x)_i=\varphi(c_i^T x), \; c_i\in R^n, i=1,\dots m$. Here $\varphi (t)$ is twice differentiable scalar function, \[ |\varphi'(t)| \geq \mu_\varphi> 0, \;\; |\varphi''(t)|\leq L_\varphi, \;\; \forall t. \]

It has been explained in Section 5.1 that the special structure of the problem allows to get much sharper results.

Here we restrict ourselves with the single example to demonstrate how the methods work for medium-size problems ($n=40, m=21$). The equations have special structure as in Section 5.1:

\[ P_i(x) = \varphi(c_i^T x - b_i) - y_i, \;\; x \in \R^n, \;y \in \R^m,\] where 
\[ 
\varphi(t) = \frac{t}{1 + e^{-|t|}}, \;\; \varphi'(t) = \frac{1 + (1 + |t|)e^{-|t|} }{(1 + e^{-|t|})^2}. 
\]

Matrix $C$ with rows $c_i$, vectors $b, y$ were generated randomly. For function $\varphi(t)$ we have $\mu_{\varphi}=\max_t \varphi'(t)\ge 0.5, \; L_\varphi= \max_t|\varphi''(t)|\le 2$ for all $t$. 
Thus if we do not pay attention to the special structure of the problem we have $\mu \geq 0.5 \sigma_{\min}(C), \;L \leq 2 \sigma_{\max}(C)$. On the other hand if we take into account the structure we can replace $\frac{\mu^2}{L}$ ($= 0.0012$ in the example) in Algorithm~1 (see Subsection~\ref{ssec:newton-L-known}) with $\frac{\mu_{\varphi}^2}{L_\varphi} = 0.125$. 

The results of simulations are as follows. When we apply Algorithm 1 with values $L, \mu$, it requires 6000 iterations to achieve the accuracy $||P(x^k)||<10^{-12}$, while the same algorithm with $L_{\varphi}, \mu_{\varphi}$ requires just 70 iterations. The similar result holds for Algorithm 3: the version exploiting $L$ requires 30 iterations, exploiting $ {L_\varphi}$ --- just 5 iterations. All algorithms which are not based on information on these constants (pure Newton, adaptive Algorit{\color{red}h}m 2) also converge in 5 iterations.

These results demonstrate how sensitive can be the proposed algorithms to {\color{red}a priori} data and to the special structure of equations.

\subsection{Example 4. }
This test is devoted to the underdetermined systems of equations and, specifically, sparsity property, arising in optimal control problems.
The behavior of a pendulum with force control $u$ is given by the second-order differential equation $$\ddot{\phi} + \alpha \dot{\phi} + \beta \sin \phi = u.$$ Given some initial condition $\phi(0), \dot{\phi}(0)$, the goal is to drive the pendulum to the specified terminal position and angular speed $[\phi(T), \dot{\phi}(T)]^T = b \in \R^2$ for the fixed time $T$. The secondary goal is to have sparse control and the least control capacity $\int_{0}^T |u(t)| dt$.

The model was discretized on the interval $[0, T]$. The discretized control $U$ has dimension $N = T/h - 1$, where $h$ is the discretization step.
The problem is to solve two equations $[\phi_d(T, U), \psi_d(T, U)]^T = b$, where $\phi_d, \psi_d$ are the discrete counterparts of $\phi, \dot{\phi}$, in $N$ dimensional variable $U$. 

The problem was solved by exploiting Algorithm~2 with specific choice of norm, namely - $\ell_1$-norm. First, it represents a discretized control capacity ($\|u\| = \sum_{i=0}^N |U^{(i)}|$). Second, it is known for its property of finding sparse solution. The initial approximation was $U=0$.
Newton method converges in 3 steps, resulting in 5 non-zero components of control (i.e. the control should be applied at 5 time instants only). Moreover, the first Newton step reveals 2 components (time instants), which are sufficient to get to the goal, see Figure~\ref{fig:exp-4-optimal-control}. Thus 2-impulse control (with impulses at $t=98$ and $t=153$) solves the problem.

Details on the simulation can be found in \cite{Polyak-Tremba-optimal-control}.

\begin{figure}
\caption{Solution with two-impulse control{\color{red}, \cite{Polyak-Tremba-optimal-control}}.}
\includegraphics[width=1.0\textwidth]{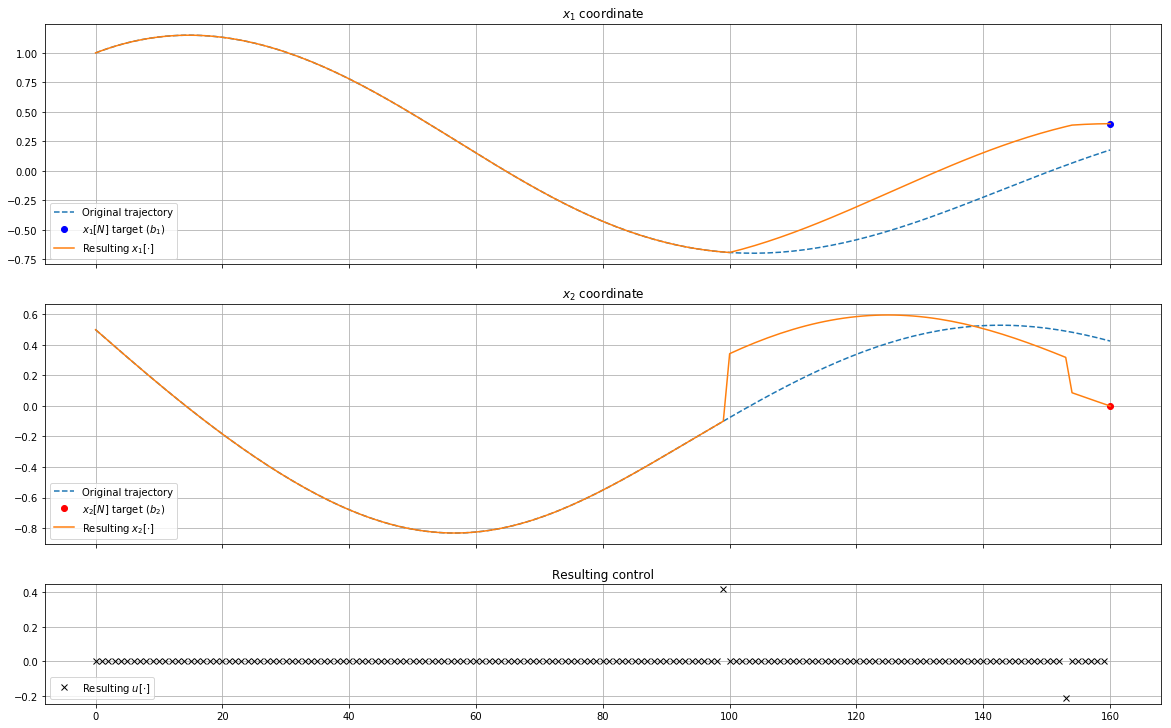}
\label{fig:exp-4-optimal-control}
\end{figure}

\section{Conclusions and future research}
\label{sec:conclusion}

New solvability conditions for under-determined equations (with wider solvability set) are proposed. The algorithms for finding a solution are easy to implement, they combine weaker assumptions on initial approximations and fast convergence rate. No convexity assumptions are required. 
The algorithms have large flexibility in using prior information, various norms and problem structure. It is worth mentioning that we do not try to convert the problem into optimization one. Combination of damped/pure Newton method is a contribution for solving classic $n = m$ problems as well.

There are numerous directions for future research.
\begin{enumerate}
\item We suppose that the auxiliary optimization problem for finding direction $z^k$ is solved exactly. Of course an approximate solution of the sub-problem suffices.
\item The algorithms provide a solution of the initial problem which is not specified a priori. Sometimes we are interested in the solution closest to $x^0$, i.e. $\min_{P(x) = 0} \|x - x^0\|$. An algorithm for this purpose is of interest.
\item More general theory of structured problems (Section 5.1) is needed.
\item It is not obvious how to introduce regularization techniques into the algorithms. 
\end{enumerate}

\section*{{\color{red}Acknowledgments}}
The authors thank Yuri Nesterov and Alexander Ioffe for helpful discussions and references; the comments of the anonymous reviewers are highly acknowledged.

\section*{Funding}
This work was supported by the Russian Science Foundation under Grant 16-11-10015.

\end{document}